\documentclass[10pt,reqno]{amsart}
\usepackage{amsmath}
\usepackage{amssymb}
\usepackage{amsthm}
\usepackage[usenames]{color}
\usepackage{eepic,epic}
\usepackage{xcolor}

\newtheorem{thm}{Theorem}[section]
\newtheorem{cor}[thm]{Corollary}
\newtheorem{lem}[thm]{Lemma}
\newtheorem{prop}[thm]{Proposition}

\newtheorem{theoremalpha}{Theorem}

\theoremstyle{definition}

\theoremstyle{remark}
\newtheorem{rem}[thm]{Remark}
\numberwithin{equation}{section}

\begin{document}
\title[Splitting method for nonlinear Schr\"odinger equation]
{Time splitting method for nonlinear Schr\"odinger equation with rough initial data in $L^2$}

\author[Choi, Kim and Koh]{Hyung Jun Choi, Seonghak Kim and Youngwoo Koh}

\address{School of Liberal Arts, Korea University of Technology and Education, Cheonan 31253, Republic of Korea}
\email{hjchoi@koreatech.ac.kr}

\address{Department of Mathematics, College of Natural Sciences, Kyungpook National University, Daegu 41566, Republic of Korea}
\email{shkim17@knu.ac.kr}

\address{Department of Mathematics Education, Kongju National University, Kongju 32588, Republic of Korea}
\email{ywkoh@kongju.ac.kr}

\thanks{The authors equally contributed to this work}

\maketitle
\begin{abstract}
We establish convergence results related to the operator splitting scheme on the Cauchy problem for the nonlinear Schr\"odinger equation with rough initial data in $L^2$,
    $$
    \left\{
    \begin{array}{ll}
      i\partial_t u +\Delta u = \lambda |u|^{p} u, & (x,t) \in \mathbb{R}^d \times \mathbb{R}_+, \\
      u (x,0) =\phi (x), & x\in\mathbb{R}^d,
    \end{array}
    \right.
    $$
where $\lambda \in \{-1,1\}$ and $p >0$.
While the Lie approximation $Z_L$ is known to converge to the solution $u$ when the initial datum $\phi$ is sufficiently smooth,
the convergence result for rough initial data is open to question.
In this paper, for rough initial data $\phi\in L^2 (\mathbb{R}^d)$, we prove the $L^2$ convergence of the filtered Lie approximation $Z_{flt}$ to the solution $u$ in the mass-subcritical range, $0< p < \frac{4}{d}$.
Furthermore, we provide a precise convergence result for radial initial data $\phi\in L^2 (\mathbb{R}^d)$.
\end{abstract}


\section{Introduction}

Consider the Cauchy problem for the nonlinear Schr\"odinger equation in $\mathbb{R}^{d+1}$:
    \begin{equation}\label{main-equation}
    \left\{\begin{array}{ll}
    i\partial_t u + \Delta u = \lambda |u|^p u, & (x,t) \in \mathbb{R}^d \times \mathbb{R}_+, \\
    u(x,0)= \phi (x), & x\in\mathbb{R}^d,
    \end{array} \right.
    \end{equation}
where $d\in\mathbb{N}$, $\lambda \in \{-1,1\}$, $p>0$, and $\mathbb{R}_+ = (0,\infty)$.
The nonlinear Schr\"odinger equations appear in various models of quantum mechanics. Regarding history, one may refer to some literatures \cite{C,MN,SS,T}.

The well-posedness theory of \eqref{main-equation} in the Sobolev space $H^s (\mathbb{R}^d)$ is currently well understood.
We briefly summarize certain facts on the subcritical case of \eqref{main-equation}.
For $0 < p < \frac{4}{(d-2s)_{+}}$ with $s\geq 0$,
let us take an initial datum $\phi\in H^s (\mathbb{R}^d)$, where $(a)_{+}:=\max\{a,0\}$ for $a\in\mathbb{R}$.
Then there exist a maximal time $T_{\max}=T(d,p,\lambda,s)>0$ and a unique solution $u \in C\left([0,T_{\max}); H^s (\mathbb{R}^d)\right)$ to \eqref{main-equation} satisfying
    \begin{equation}\label{NLS-result}
    \sup_{0 \leq t \leq T} \|u(t)\|_{H^s (\mathbb{R}^d)} + \|u\|_{L^q([0,T]; W^{s,r} (\mathbb{R}^d))} \leq C(d,p,T,\phi)
    \end{equation}
for all $0<T<T_{\max}$ and Schr\"odinger admissible pairs $(q,r)$.
Here, a pair $(q,r)$ is called \emph{(Schr\"odinger) admissible} if
    \begin{equation}\label{admissible}
    \frac{2}{q} + \frac{d}{r} = \frac{d}{2}, \quad 2\leq q,r \leq \infty, \quad \mbox{and}\quad (q,r,d)\neq(2,\infty,2).
    \end{equation}
Furthermore, we have $T_{\max}=\infty$ if $0 < p < \frac{4}{d}$ (called \textit{mass-subcritical}) or $\lambda=1$ (called \textit{defocusing}).
For details, we refer the reader to Sections 4 and 5 of Cazenave \cite{C}.\\

In this paper, we are concerned with operator splitting schemes, which are widely used in the numerical computation of semilinear equations such as in \eqref{main-equation}. The basic concept of the schemes is to divide the problem \eqref{main-equation} into linear and nonlinear flows and then to apply the two flows to the previous data in a fixed order.
We set $S(t) \phi$ as the solution of the linear Schr\"odinger propagation:
    $$
    \left\{ \begin{aligned}  i\partial_t u & = -\Delta u, & (x,t) \in \mathbb{R}^d \times \mathbb{R}_+,
    \\
    u(x,0)&= \phi (x),& x \in \mathbb{R}^d,
    \end{aligned}
    \right.
    $$
which admits the Fourier multiplier formula $S(t) \phi = e^{it \Delta} \phi$.
Also, we define $N(t) \phi$ to be the solution of the flow:
    $$
    \left\{ \begin{aligned}  i\partial_t u & = \lambda |u|^p u,& (x,t) \in \mathbb{R}^d \times \mathbb{R}_+,
    \\
    u(x,0)&= \phi (x),& x \in \mathbb{R}^d.
    \end{aligned}
    \right.
    $$
Thus, we have $N(t) \phi = \exp (-i t \lambda |\phi|^p) \phi$.
We now split the flow of \eqref{main-equation} into nonliear flow $N(t)$ and linear flow $S(t)$ with a small switching time. Specifically, for a fixed time interval $[0,T]$ and small $\tau>0$, we consider an approximation $Z(n\tau)$, based on an $n$-time iteration of $N(\tau)$ and $S(\tau)$.
For the combination of the two flows, there are two popular approximations, namely, the Lie approximation $Z_L$,
    $$
    Z_L(n \tau) \phi = \big( S(\tau) N(\tau) \big)^n \phi,\quad 0 \leq n \tau \leq T ,
    $$
and the Strang approximation $Z_S$,
    $$
    Z_S(n \tau) \phi = \big( N(\tau/2) S(\tau) N(\tau/2) \big)^n \phi,\quad 0 \leq n \tau \leq T.
    $$
\

In 2008, Lubich \cite{L} first established the $L^2$ convergence of the Strang approximation $Z_S$ to the solution $u$ of \eqref{main-equation} for the cubic nonlinear Schr\"odinger equation in $\mathbb{R}^{3}$.
More precisely, for $d=3$, $p=2$, $\phi \in H^4 (\mathbb{R}^3)$, and $T<T_{\max}$,
the Strang approximation $Z_S$ satisfies
    \begin{equation}\label{Lubich-result}
    \max_{0 \leq n \tau \leq T} \|Z_S(n \tau)\phi - u (n\tau)\|_{L^2 (\mathbb{R}^3)}
    \leq C(T, \phi)\tau^{2}  .
    \end{equation}
Subsequently, for $1\leq d \leq3$ and $p=2$, Eilinghoff-Schaubelt-Schratz \cite{ESS} extended the result to the wider class of initial data, namely, $\phi \in H^{2+2\epsilon}(\mathbb{R}^d)$ $(\epsilon >0)$,
with convergence rate $\tau^{1+\epsilon}$ in the right hand side of \eqref{Lubich-result}.

Regarding the Lie approximation $Z_L$, since the nonlinear flow $N(t)$ satisfies
    $$
    N(\tau)\phi =N(\tau/2)N(\tau/2)\phi
    $$
in the case of the nonlinear Schr\"odinger equation \eqref{main-equation},
the error estimate of the Lie approximation $Z_L$ is essentially similar to that of the Strang approximation $Z_S$.
Eilinghoff-Schaubelt-Schratz \cite{ESS} recently obtained
the $L^2$ convergence of the Lie approximation $Z_L$ for initial data $\phi \in H^{2} (\mathbb{R}^d)$.
(Earlier, Besse-Bid\'egaray-Descombes \cite{BBD} presented a convergence result in $d=2$ for a general Lipschitz nonlinearity.)
Precisely, for $1\leq d \leq3$, $p=2$, $\phi \in H^{2} (\mathbb{R}^d)$ and $T<T_{\max}$,
they showed that
    $$
    \max_{0 \leq n \tau \leq T} \|Z_L(n \tau)\phi - u (n\tau)\|_{L^2 (\mathbb{R}^d)} \leq   C(T,\phi)\tau.
    $$
Our main objective is to investigate the convergence of the approximations to the solution $u$ of \eqref{main-equation} for a wide class of rough initial data.

The problem of reducing the regularity of initial data is challenging.
To overcome this difficulty, various modified versions of the splitting scheme have been investigated.
Ignat \cite{I2} considered a filtered Lie approximation $Z_{flt}$,
    \begin{equation}\label{flt-Lie_approx}
    Z_{flt} (n\tau)\phi = \big( P(\tau)S(\tau) N(\tau) \big)^n P(\tau) \phi,
    \end{equation}
where $P(\tau)$ denotes the frequency localized multiplier, given by
    \begin{equation}\label{freq-loc-oper}
    \widehat{P(\tau) \phi} (\xi)
    = \chi (\tau^{\frac{1}{2}}\xi) \widehat{\phi}(\xi) ,\quad \xi \in \mathbb{R}^d.
    \end{equation}
Here, $\chi \in C^{\infty}(\mathbb{R}^d)$ is a cut-off function, supported in $B(0,2)$, such that $\chi \equiv 1$ in $B(0,1)$.
For $1\leq d\leq3$, $1\leq p < \frac{4}{d}$, $\phi \in H^2 (\mathbb{R}^d)$, and $T<\infty$,
Ignat \cite{I2} proved the $L^2$ convergence as
    \begin{equation}\label{Ignet-result}
    \max_{0 \leq n \tau \leq T} \| Z_{flt} (n \tau)\phi - u (n\tau)\|_{L^2 (\mathbb{R}^d)}
    \leq  C(d,p,T,\phi)\tau.
    \end{equation}
Later, this result was extended by Choi-Koh \cite{CK} to all $H^1$ \textit{energy-subcritical} cases.
Precisely, for $1\leq d \leq 5$, $0< p < \frac{4}{(d-2)_{+}}$, $\phi \in H^1 (\mathbb{R}^d)$, and $T<T_{\max}$,
the estimate \eqref{Ignet-result} holds with the convergence rate $\tau$ replaced by $\tau^{\frac{1}{2}}$.
Furthermore, for $d=1$, $p=2$, $0<s<1$, and $T<\infty$, Ostermann-Rousset-Schratz \cite{ORS} showed that
    \begin{equation}\label{ORS-result}
    \max_{0 \leq n \tau \leq T} \| Z_{flt} (n \tau)\phi - u (n\tau)\|_{L^2 (\mathbb{R})}
    \leq C(d,p,T,\phi)\tau^{\frac{s}{2}}
    \end{equation}
for all $\phi\in H^{s+\epsilon}(\mathbb{R})$ with arbitrary $\epsilon>0$.

As we can see in \eqref{Lubich-result}, \eqref{Ignet-result}, and \eqref{ORS-result},
for initial data $\phi \in H^s (\mathbb{R}^d)$, the order $\frac{s}{2}$ of convergence rate  might be understood as the natural order barrier for the Strang and Lie splitting schemes.
However, Ostermann-Rousset-Schratz \cite{ORS2} broke the natural order barrier $\frac{s}{2}$ using a modified approximation $Z_{ORS}$ that is based on the Lie approximation,
    $$
    Z_{ORS}((n+1)\tau)
    = S(\tau) \bigg( Z_{ORS}(n\tau) - i\tau P(\widetilde{\tau}) \big( P(\widetilde{\tau}) Z_{ORS}(n\tau) \big)^2 \frac{S(-2\tau)-1}{-2i\tau\Delta} P(\widetilde{\tau}) \overline{Z_{ORS}(n\tau)} \bigg) P(\widetilde{\tau}),
    $$
where $\widetilde{\tau} =\tau^{C(d)}$ for some $C(d) >0$.
Precisely, for $1\leq d \leq3$, $p=2$, $\phi \in H^1 (\mathbb{R}^d)$, and $T<T_{\max}$,
$Z_{ORS}(n\tau)\phi$ converges to the solution $u$ with convergence rate of order $\frac{1}{2}+\frac{5-d}{12}$.
Later, subsequent works \cite{LW,OY,WY,W,ORS3} also broke the natural order barrier $\frac{s}{2}$ in their different settings.
Recently, for the rough initial data $\phi \in H^{s} (\mathbb{T})$ with $0<s<1$, Wu \cite{W} obtained the $L^2$ convergence of the modified Strang approximation
    $$
    Z_W((n+1)\tau)= S(\tau/2) \widetilde{N}(\tau) S(\tau/2) \Big( P(\widetilde{\tau})+ e^{-2\pi i\lambda M_0 \tau}(1-P(\widetilde{\tau}) \Big) Z_W(n\tau)
    $$
where $\widetilde{N}(\tau) \phi = \exp (-i t \lambda |P(\widetilde{\tau})\phi|^2) \phi$, $\widetilde{\tau} =\tau^{\frac{2}{4+s}}$ and $M_0$ the mass of the initial data.
Precisely, for $d=1$, $p=2$ and $\phi \in H^{s} (\mathbb{T})$ with $0<s<1$,
$Z_{W}(n\tau)\phi$ converges to the solution $u$ with convergence rate of order $\frac{4s}{4+s}$
which obtained by the analysis separating high-low frequency components.

\

In this paper, we consider the convergence of the filtered Lie approximation $Z_{flt}$ for rough initial data $\phi\in L^2(\mathbb{R}^d)$.
Previous approaches result in the estimate,
    \begin{equation}\label{fail_L2}
    \max_{0 \leq n \tau \leq T} \|Z_{flt}(n \tau)\phi - u (n\tau)\|_{L^2 (\mathbb{R}^d)} \leq C(T,\phi),
    \end{equation}
that fails to give a convergence of $Z_{flt}$ for $\phi \in L^2 (\mathbb{R}^d)$.
Thus, to deal with rough initial data in $L^2 (\mathbb{R}^d)$,
we need a more careful analysis than the previous methods.

The first result of this paper pertains to the convergence of the filtered Lie approximation $Z_{flt}$ for initial data in $L^2 (\mathbb{R}^d)$.

\begin{thm}\label{main-thm}
Let $1\leq d\leq3$ and $0< p < \frac{4}{d}$. Then, for any $\phi \in L^2 (\mathbb{R}^d)$ and $T>0$,
    \begin{equation}\label{main-result}
    \max_{0 \leq n \tau \leq T} \big\| Z_{flt} (n \tau)\phi - u (n\tau) \big\|_{L^2 (\mathbb{R}^d)} \rightarrow 0 \quad\mbox{as}\quad \tau \rightarrow 0 .
    \end{equation}
\end{thm}

To derive \eqref{main-result},
we obtain some precise local-in-time estimates in Section \ref{sec_local},
which are the main contribution of the paper.
We then complete the proof of Theorem \ref{main-thm} in Section \ref{sec_global}.
The condition that $p < \frac{4}{d}$ in Theorem \ref{main-thm} follows naturally from the \textit{mass conservation law}.

\begin{rem}
Theorem \ref{main-thm} can be naturally extended to higher dimensions $d\geq 4$ with a technical condition $p\leq \frac{2}{d-2}$. This condition $p\leq \frac{2}{d-2}$ might be removable, but we focus on only physical dimensions $1\leq d\leq 3$ in this paper.
\end{rem}

Additionally, we consider the case of radial initial data in $L^2(\mathbb{R}^d)$.
In this case,
we can make use of several extended results of the Strichartz estimates that enable us to obtain a tighter bound than \eqref{main-result} as follows.

\begin{thm}\label{main-thm3}
Let $2\leq d\leq3$ and $0< p < \frac{4}{d}$. Then, for any radial function $\phi \in L^2 (\mathbb{R}^d)$ and $T>0$,
    \begin{equation}\label{main-result5}
    \begin{aligned}
    &\max_{0 \leq n \tau \leq T} \big\| Z_{flt}(n \tau)\phi - u (n\tau) \big\|_{L^2 (\mathbb{R}^d)} \\
    &\quad\leq C_{d,p} \exp\big( C_{d,p}T \| \phi\|_{L^2(\mathbb{R}^d)}^{\frac{4p}{4-dp}} \big) \\
    &\qquad \times\Big( \| \phi - P(\widetilde{\tau})\phi \|_{L^2(\mathbb{R}^d)}
    + \widetilde{\tau}^{\frac{4-dp}{8}} T \| \phi \|_{L^2(\mathbb{R}^d)}^{C_{d,p}}  + \Big(\frac{\tau}{\widetilde{\tau}}\Big)^{\frac{1}{2}} \big( \|\phi\|_{L^2(\mathbb{R}^d)} +\|\phi\|_{L^2(\mathbb{R}^d)}^{p+1}\big) \Big)
    \end{aligned}
    \end{equation}
for any $0<\tau\leq \widetilde{\tau}<1$.
\end{thm}

If we take $\widetilde{\tau} =\tau^{1-\epsilon}$ $(0<\epsilon<1)$,
the right hand side of \eqref{main-result5} trivially converges to zero as $\tau\rightarrow0$.
Furthermore, the bound \eqref{main-result5} is written as an explicit expression in terms of $\tau$, $\widetilde{\tau}$, $T$, and $\phi$.
Thus, we can control the convergence estimate by the $L^2$ norm of $\phi$,  support of its Fourier transform (which relies on the oscillation of $\phi$), and time $T$.
We will give a proof of Theorem \ref{main-thm3} in Section \ref{sec_radial}.

For instance, consider the radial function $\phi_\alpha$, given by
    \begin{equation}\label{radial_exam}
    \widehat{\phi_\alpha}(\xi) = C(1+|\xi|)^{-\frac{d}{2}} \big(\log(2+|\xi|) \big)^{-\frac{1+\alpha}{2}}
    \quad\mbox{for all}\; \xi\in \mathbb{R}^d,
    \end{equation}
where $\alpha >0$.
The function $\phi_\alpha$ is a typical one that belongs to $L^2(\mathbb{R}^d)\setminus H^s(\mathbb{R}^d)$ for any $s>0$.
We can easily calculate
    $$
    \| \phi_\alpha - P(\widetilde{\tau})\phi_\alpha \|_{L^2(\mathbb{R}^d)} = C(-\log\widetilde{\tau} )^{-\frac{\alpha}{2}}
    \quad\mbox{and}\quad \| \phi_\alpha \|_{L^2(\mathbb{R}^d)}\leq C.
    $$
Taking $\widetilde{\tau}=\tau (-\log \tau)^{\alpha}$ in \eqref{main-result5},
we obtain a bound of logarithmic order,
    $$
    \begin{aligned}
    \max_{0 \leq n \tau \leq T} \big\| Z_{flt}(n \tau) \phi_\alpha - u (n\tau) \big\|_{L^2 (\mathbb{R}^d)}
    \leq C(d,p,T)\Big( \frac{-1}{\log \tau} \Big)^{\frac{\alpha}{2}-\epsilon}
    \end{aligned}
    $$
for any $0<\tau<1$ and $\epsilon>0$. We record this result in a slightly more general fashion as follows.

\begin{cor}\label{cor_result}
Let $1\leq d\leq3$ and $0< p < \frac{4}{d}$. Then, for any $T>0$ and radial function $\phi \in L^2 (\mathbb{R}^d)$ with a decay of its Fourier transform,
    $$
    |\widehat{\phi}(\xi)| \leq C|\xi|^{-\frac{d}{2}} \big(\log|\xi| \big)^{-\frac{1+\alpha}{2}}
    \quad\mbox{if}\quad |\xi|\gg 1,
    $$
for some $\alpha >0$, the filtered Lie approximation $Z_{flt}$ satisfies
    $$
    \begin{aligned}
    \max_{0 \leq n \tau \leq T} \big\| Z_{flt}(n \tau)\phi - u (n\tau) \big\|_{L^2 (\mathbb{R}^d)}
    \leq C(d,p,T,\phi)\Big( \frac{-1}{\log\tau} \Big)^{\frac{\alpha}{2}-\epsilon}
    \end{aligned}
    $$
for any $0<\tau<1$ and $\epsilon>0$.
\end{cor}

\begin{rem}
We may guess that Corollary \ref{cor_result} also holds for a more general condition on $\phi$.
In this point of view,
it will be an interesting question on the convergence rate when $\phi$ belongs to an intermediate space between $L^2(\mathbb{R}^d)$ and $H^{\epsilon}(\mathbb{R}^d)$ $(\epsilon>0)$.
For instance, we can consider an initial datum $\phi$ that belongs to the logarithmic Sobolev space $H_{\log}^s (\mathbb{R}^d)$ $(s>0)$ (see \cite{ET}) with its norm
    $$
    \| \phi \|_{H_{\log}^{s} (\mathbb{R}^d)}
    =\Big( \int_{\mathbb{R}^d} \big( \log(2+|\xi|) \big)^{2s} |\widehat{\phi}(\xi)|^2 d\xi \Big)^{\frac{1}{2}} < \infty.
    $$
(In particular, one can check that the function $\phi_{\alpha}$ in \eqref{radial_exam} belongs to $H_{\log}^{s}(\mathbb{R}^d)$ for any $0< s <\alpha$.)
In this case, we may speculate that
    $$
    \max_{0 \leq n \tau \leq T} \big\| Z_{flt}(n \tau)\phi - u (n\tau) \big\|_{L^2 (\mathbb{R}^d)}
    \leq \Big( \frac{-1}{\log\tau} \Big)^{\frac{s}{2}} C(d,p,T,\phi).
    $$
However, such an estimate remains open since a local well-posedness theory of (\ref{main-equation}) in the logarithmic Sobolev space $H_{\log}^{s}(\mathbb{R}^d)$ has not been established.
\end{rem}

Closing this section, we introduce some related topics.
There are interesting variants of the splitting methods for the higher-order methods \cite{HO,T1,TCN},
for the spectral discrete spaces \cite{GL1,GL2,I1,IZ1,IZ2},
and for the numerical experiment \cite{LM}.
Furthermore, the splitting methods have been studied for various equations such as
the linear Schr\"odinger equation with potentials \cite{JL},
Schr\"odinger equation with Hartree type nonlinearity \cite{L},
cubic nonlinear Schr\"odinger equation on torus \cite{KOS,OS,OY,LW,WY,W,ORS3},
KdV equation \cite{HKR,HKRT,LWY,RS2,WZ,WZ2,LWZ},
Burgers' equation \cite{HLR},
Dirac equation \cite{SWZ},
and various parabolic equations \cite{AO,BCCM,CCDV,CCK,D,DDDLLM,DR,FOS,RS}.\\

\noindent \textbf{Notations}

If a constant depends on some other values, we denote it by $C(d,p,T,\phi)$ or $C_{d,p,T,\phi}$ (depending on dimension $d$, nonlinear exponent $p$, time  $T$, and initial datum $\phi$). Their values may change from line to line.
Also, for an interval $I\subset\mathbb{R}$, we often write $\|\cdot\|_{\ell^q (I;L^r)}$ and $\|\cdot\|_{L^q (I;L^r)}$ for $\|\cdot\|_{\ell^{q} (I; L^r(\mathbb{R}^d))}$ and $\|\cdot\|_{L^{q} (I; L^r(\mathbb{R}^d))}$, respectively. For any number $q\in[1,\infty],$ its H\"older conjugate is denoted by $q'\in[1,\infty].$


\section{Preliminaries}

In this section, we summarize some basic theories that will be used throughout this paper.
Prior to stating basic theories, we define certain discrete function spaces.
Let $I\subset\mathbb{R}$ be an interval, $\tau\in(0,1)$, $q\in[1,\infty],$ and
$J_\tau=\{n\tau\,:\,n\in\mathbb{Z},n\tau\in I\}$.
Let $X$ be a Banach space with its norm $\|\cdot\|_{X}$. Then we denote by $\ell^q (I;X)=\ell^q_\tau (I;X)$ the space of functions $f:J_\tau\to X$ with norm $\|f(n\tau)\|_{\ell^q (I;X)}<\infty$, where

    $$
    \|f(n\tau)\|_{\ell^q (I;X)}
    = \left\{\begin{array}{ll}
        \big( \tau \sum_{n\tau \in I} \|f(n \tau)\|_X^q \big)^{\frac{1}{q}} & \mbox{for}\quad 1\le q<\infty, \\
        \sup_{n\tau\in I}  \|f(n \tau)\|_X & \mbox{for}\quad q=\infty.
      \end{array}\right.
    $$

Now, we state the local well-posedness theory of $u$ in \eqref{main-equation} and $Z_{flt}$ as follows.

\begin{theoremalpha}[Local well-posedness of $u$ and $Z_{flt}$]\label{thm_LWP}
Let $d\geq1$, $0 < p < \frac{4}{d}$, and $\phi \in L^2 (\mathbb{R}^d)$.
Let $u$ be the solution of \eqref{main-equation} and $Z_{flt}$ the filtered Lie approximation, defined by \eqref{flt-Lie_approx}.
Then there exists a constant $c_{d,p}>0$ such that, with $T_{loc} = c_{d,p} \|\phi\|_{L^2(\mathbb{R}^d)}^{-\frac{4p}{4-dp}}$ and $I_0=[0,T_{loc}]$,
    \begin{equation}\label{NLS-bdd}
    \| u \|_{L^q(I_0; L^{r} (\mathbb{R}^d))} +\| P(\tau)u(n\tau) \|_{\ell^q(I_0; L^{r} (\mathbb{R}^d))}
    \leq C_{d} \|\phi\|_{L^2 (\mathbb{R}^d)}
    \end{equation}
and
    \begin{equation}\label{Lie-bdd}
    \| Z_{flt}(n\tau)\phi \|_{\ell^{q} (I_0; L^{r}(\mathbb{R}^d))}
    \leq C_{d} \|\phi\|_{L^2 (\mathbb{R}^d)}
    \end{equation}
for any $0<\tau<1$ and admissible pair $(q,r)$.
\end{theoremalpha}

The proof of \eqref{NLS-bdd} in Theorem \ref{thm_LWP} is standard.
For the estimate of the time-continuous norm $\|u \|_{L^q(I_0; L^{r} )}$, see Sections 4 and 5 of Cazenave \cite{C}.
The estimate of the time-discrete norm $\| P(\tau)u(n\tau) \|_{\ell^q(I_0; L^{r} )}$ can be obtained by the same argument regarding the time-continuous norm.
For readers' convenience, we give a proof of the time-discrete norm $\| P(\tau)u(n\tau) \|_{\ell^q(I_0; L^{r} )}$ in Appendix.
Also, the local well-posedness \eqref{Lie-bdd} of $Z_{flt}$ is provided in Theorem 1.1 of Ignat \cite{I2}.

Second, we state the Strichartz estimates for the linear Schr\"odinger operator $S(t)$ $(t\in\mathbb{R})$.

\begin{theoremalpha}[Time-continuous and time-discrete Strichartz estimates]\label{thm-str}
(i) There exists $C_{d}>0$ such that
    \begin{equation}\label{Str_homo}
    \|S(t) \phi \|_{L^q (\mathbb{R}; L^r (\mathbb{R}^d))}
    + \|S(n\tau)P(\tau) \phi \|_{\ell^q (\mathbb{R}; L^r (\mathbb{R}^d))}
    \leq C_{d} \|\phi\|_{L^2 (\mathbb{R}^d)}
    \end{equation}
for all $\phi\in L^2 (\mathbb{R}^d)$, $0<\tau<1$, and admissible pairs $(q,r)$.

(ii) There exists $C_{d}>0$ such that if $(q,r)$ and $(\widetilde{q}, \widetilde{r})$ are admissible pairs with $(q,\widetilde{q})\ne(2,2)$ and $0<\tau<1$, then
    \begin{equation}\label{Str_inhomo1}
    \Big\| \int_{s< t} S(t-s) F(s) ds \Big\|_{L^q (\mathbb{R}; L^{r}(\mathbb{R}^d))}
    + \Big\| \int_{s < n \tau} S(n \tau -s) P(\tau) F(s) ds \Big\|_{\ell^q (\mathbb{R}; L^r (\mathbb{R}^d))}
    \leq C_{d} \| F\|_{L^{\widetilde{q}'} (\mathbb{R}; L^{\widetilde{r}'} (\mathbb{R}^d))}
    \end{equation}
    for all $F\in L^{\widetilde{q}'}(\mathbb{R}; L^{\widetilde{r}'}(\mathbb{R}^d))$, and
    \begin{equation}\label{Str_inhomo2}
    \bigg\| \tau \sum_{k=-\infty}^{n-1}S((n-k) \tau) P(\tau) F(k\tau)\bigg\|_{\ell^q (\mathbb{R}; L^r (\mathbb{R}^d))}
    \leq C_{d} \|F(n\tau)\|_{\ell^{\widetilde{q}'}(\mathbb{R}; L^{\widetilde{r}'} (\mathbb{R}^d))}
    \end{equation}
for all $F\in \ell^{\widetilde{q}'}(\mathbb{R}; L^{\widetilde{r}'}(\mathbb{R}^d))=\ell^{\widetilde{q}'}_\tau(\mathbb{R}; L^{\widetilde{r}'}(\mathbb{R}^d))$.
\end{theoremalpha}

For the homogeneous estimate of the time-continuous norm $\|S(t) \phi \|_{L^q (\mathbb{R}; L^r )}$,
Strichartz \cite{S} first proved \eqref{Str_homo} for $q=r$; then it was extended by Keel-Tao \cite{KT} to all admissible pairs $(q,r)$, including the endpoint $(q,r)=(2,\frac{2d}{d-2})$.
The inhomogeneous estimate  \eqref{Str_inhomo1} of the time-continuous norm $\big\| \int_{s< t} S(t-s) F(s) ds \big\|_{L^q (\mathbb{R}; L^{r})}$ follows from the duality argument and Christ-Kiselev lemma \cite{CK2}.

The proof of the estimate \eqref{Str_homo} of the time-discrete norm $\|S(n\tau) P(\tau) \phi \|_{\ell^q (\mathbb{R}; L^r )}$ is essentially the same as the time-continuous one.
We refer the readers to Theorem 2.1 and Lemma 4.5 in Ignat \cite{I2}.

Next, we give a basic multiplier theory in harmonic analysis.
The lemma below follows from classical Bernstein's inequality.

\begin{lem}
Let $P(\tilde{\tau})$ be the frequency localized multiplier, defined as in \eqref{freq-loc-oper}.
Then, for any $1\leq q \leq r <\infty$, $s>0$, and $\phi:\mathbb{R}^d\rightarrow\mathbb{C}$, we have
    \begin{equation}\label{multiplier_4}
    \big\| \phi - P(\widetilde{\tau}) \phi \big\|_{L^r (\mathbb{R}^d)}
    \leq C \widetilde{\tau}^{\frac{s}{2}} \| (-\Delta)^{\frac{s}{2}} \phi\|_{L^r (\mathbb{R}^d)} ,
    \end{equation}
    \begin{equation}\label{multiplier_3}
    \| P(\widetilde{\tau}) \phi\|_{L^r (\mathbb{R}^d)}
    \leq C \|\phi\|_{L^r (\mathbb{R}^d)} ,
    \end{equation}
    \begin{equation}\label{multiplier_2}
    \big\| (-\Delta)^{\frac{s}{2}} P(\widetilde{\tau}) \phi \big\|_{L^r (\mathbb{R}^d)}
    \leq C \widetilde{\tau}^{-\frac{s}{2}} \| \phi \|_{L^r (\mathbb{R}^d)},
    \end{equation}
and
    \begin{equation}\label{multiplier_1}
    \| P(\widetilde{\tau}) \phi \|_{L^{r}(\mathbb{R}^d)}
    \leq C \widetilde{\tau}^{\frac{d}{2} \left( \frac{1}{r}-\frac{1}{q}\right)} \| \phi\|_{L^{q}(\mathbb{R}^d)}.
    \end{equation}
\end{lem}

%
Lastly, we give basic pointwise inequalities, following from the Mean Value Theorem. (See Lemma 4.2 in \cite{I2}.)

\begin{lem}
For any $p \in (0, \infty)$, there exists a constant $c_p>0$ such that
    \begin{equation}\label{MVT-1}
    \left| \frac{N(\tau) - I}{\tau} v - \frac{N(\tau) - I}{\tau}w \right|
    \leq c_p|v-w|\big(|v|^p + |w|^p\big)
    \end{equation}
and
    \begin{equation}\label{MVT-2}
    \bigg|\frac{N(\tau) - I}{\tau} v -i\lambda |v|^p v \bigg|
    \leq c_p \tau |v|^{2p+1}
    \end{equation}
for all $0<\tau<1$ and $v, w \in \mathbb{C}$.
\end{lem}


\section{Local-in-time error estimates}\label{sec_local}

In this section, we derive local-in-time error estimates of the difference between $P(\tau)u(n\tau)$ and $Z_{flt}(n\tau)\phi$.
Proposition \ref{local_main_thm} below constitutes the main part of our arguments.

\begin{prop}\label{local_main_thm}
Let $1\leq d\leq3$, $0< p < \frac{4}{d}$, and $\phi\in L^2(\mathbb{R}^d)$.
Then there exists a constant $c_{d,p}>0$ such that, with $T_{loc}=c_{d,p} \|\phi\|_{L^2(\mathbb{R}^d)}^{-\frac{4p}{4-dp}}$  and $I_0= [0,T_{loc}]$,
    \begin{equation}\label{key_estimate}
    \begin{aligned}
    &\big\| Z_{flt}(n\tau)\phi - P(\tau) u(n\tau) \big\|_{\ell^{q}(I_0; L^r(\mathbb{R}^d))} \\
    &\quad\leq \big\| P(\tau) u(n\tau) - P(\widetilde{\tau})u(n\tau) \big\|_{\ell^{q_0}(I_0; L^{r_0}(\mathbb{R}^d))} + \big\| u - P(\widetilde{\tau})u \big\|_{L^{q_0}(I_0; L^{r_0}(\mathbb{R}^d))}  \\
    &\qquad\quad +\Big(\frac{\tau}{\widetilde{\tau}}\Big)^{\frac{1}{2}}  \big( \|\phi\|_{L^2(\mathbb{R}^d)} +\|\phi\|_{L^2(\mathbb{R}^d)}^{p+1}\big)
    \end{aligned}
    \end{equation}
for any $0<\tau \leq \widetilde{\tau}<1$ and admissible pair $(q,r)$, where $(q_0, r_0)$ denotes the admissible pair, given by
    \begin{equation}\label{admisible_q0r0}
    \frac{1}{r_0}= \frac{1}{p+2}\quad\mbox{and}\quad \frac{1}{q_0}=\frac{dp}{4(p+2)}.
    \end{equation}
\end{prop}

As we can see in the proof below, the result of Proposition \ref{local_main_thm} is still valid when $I_0$ is replaced by $I_0'=[0,T_{loc}']$ $(0<T_{loc}'\le T_{loc})$.

\begin{proof}
Let $0<\tau \leq\widetilde{\tau}<1$ be given.
Since the solutions of \eqref{main-equation} and \eqref{flt-Lie_approx} can be written as
    \begin{equation}\label{Duhamel_u}
    P(\tau) u(n\tau)
    =S(n\tau)P(\tau)\phi +
    i\lambda \int_{0}^{n\tau} S(n \tau -s) P(\tau) \big( |u(s)|^p u(s)\big) ds
    \end{equation}
and
    \begin{equation}\label{Duhamel_Z}
    Z_{flt}(n\tau)\phi
    =S(n\tau) P(\tau) \phi +
    \tau \sum_{k=0}^{n-1} S(n \tau -k \tau) P(\tau) \frac{N (\tau) - I}{\tau} Z_{flt}(k \tau)\phi, \quad n \geq 1,
    \end{equation}
respectively (see \cite{I2} and \cite{CK}),
we can divide their difference into the following five parts:
    $$
    Z_{flt}(n\tau)\phi - P(\tau)u(n\tau)
    = \mathcal{A}_1 + \mathcal{A}_2 + \mathcal{A}_3 + \mathcal{A}_4 + \mathcal{A}_5,
    $$
where
    $$
    \begin{aligned}
    \mathcal{A}_1 := &\, \tau \sum_{k=0}^{n-1} S(n \tau -k \tau) P(\tau) \left( \frac{N(\tau) - I}{\tau} Z_{flt}(k\tau)\phi
        - \frac{N(\tau) - I}{\tau} P(\tau)u(k\tau) \right), \\
    \mathcal{A}_2 := &\, \tau \sum_{k=0}^{n-1} S(n \tau -k \tau) P(\tau) \left( \frac{N(\tau) - I}{\tau} P(\tau)u(k\tau)
        - \frac{N(\tau) - I}{\tau} P(\widetilde{\tau}) u(k\tau) \right), \\
    \mathcal{A}_3 := &\, \tau \sum_{k=0}^{n-1} S(n \tau -k \tau)P(\tau) \left( \frac{N(\tau) - I}{\tau} P(\widetilde{\tau}) u(k\tau)
        - i\lambda \big| P(\widetilde{\tau}) u(k\tau)\phi\big|^p P(\widetilde{\tau}) u(k\tau) \right), \\
    \mathcal{A}_4 := &\, i\lambda\tau \sum_{k=0}^{n-1} S(n \tau -k \tau)P(\tau) \Big( \big| P(\widetilde{\tau}) u(k\tau)\big|^p P(\widetilde{\tau}) u(k\tau) \Big)\\
    &\quad - i\lambda\int_0^{n\tau} S(n\tau -s)P(\tau) \Big( \big| P(\widetilde{\tau}) u(s) \big|^p P(\widetilde{\tau}) u(s) \Big) ds, \\
    \mathcal{A}_5 := &\, i\lambda\int_0^{n\tau} S(n\tau -s)P(\tau) \Big( \big| P(\widetilde{\tau}) u(s) \big|^p P(\widetilde{\tau})u(s)  - |u(s)|^p u(s)\Big) ds.
    \end{aligned}
    $$
Let
    $$
    T_{loc}=c_{d,p} \|\phi\|_{L^2(\mathbb{R}^d)}^{-\frac{4p}{4-dp}}
    \quad\mbox{and}\quad I_0= [0,T_{loc}],
    $$
where $c_{d,p}>0$ is a constant
(which is at most the value of  $c_{d,p}>0$ from Theorem \ref{thm_LWP})
to be determined later.
Then, by the triangle inequality, we get
    \begin{equation}\label{triangle_bdd}
    \begin{split}
    &\big\| Z_{flt}(n\tau)\phi - P(\tau) u(n\tau) \big\|_{\ell^{q}(I_0; L^r)} \\
    &\quad\leq \| \mathcal{A}_1 \|_{\ell^{q}(I_0; L^r)} + \| \mathcal{A}_2 \|_{\ell^{q}(I_0; L^r)} + \| \mathcal{A}_3 \|_{\ell^{q}(I_0; L^r)} + \| \mathcal{A}_4 \|_{\ell^{q}(I_0; L^r)} + \| \mathcal{A}_5 \|_{\ell^{q}(I_0; L^r)}
    \end{split}
    \end{equation}
for any admissible pair $(q,r)$.
Here, estimating $\mathcal{A}_4$ will be a major part of the proof.

First, to estimate $\mathcal{A}_1$, we use \eqref{MVT-1} to find that
    \begin{equation}\label{A1_MVT_bdd}
    \begin{aligned}
    &\Big| \frac{N(\tau) - I}{\tau} Z_{flt}(k\tau)\phi - \frac{N(\tau) - I}{\tau} P(\tau)u(k\tau) \Big| \\
    &\qquad\leq c_p \big|Z_{flt}(k\tau)\phi - P(\tau)u(k\tau) \big| \Big( \big|Z_{flt}(k\tau)\phi \big|^p + \big|P(\tau)u(k\tau) \big|^p \Big).
    \end{aligned}
    \end{equation}
After applying the Strichartz estimate \eqref{Str_inhomo2} to $\mathcal{A}_1$, we use \eqref{A1_MVT_bdd} and H\"older's inequality to obtain
    $$
    \begin{aligned}
    \| \mathcal{A}_1 \|_{\ell^{q}(I_0; L^r)}
    \leq & C_{d,p}  \Big\| \big| Z_{flt}(n\tau)\phi - P(\tau)u(n\tau) \big| \Big( \big|Z_{flt}(n\tau)\phi \big|^p + \big|P(\tau)u(n\tau) \big|^p \Big) \Big\|_{\ell^{q_0'} (I_0; L^{r_0'} )} \\
    \leq & C_{d,p} T_{loc}^{1-\frac{dp}{4}} \big\| Z_{flt}(n\tau)\phi - P(\tau)u(n\tau) \big\|_{\ell^{q_0} (I_0; L^{r_0})} \\
    &\times \Big(  \big\| Z_{flt}(n\tau)\phi \big\|_{\ell^{q_0} (I_0; L^{r_0})}^p + \big\| P(\tau)u (n\tau) \big\|_{\ell^{q_0} (I_0; L^{r_0})}^p \Big).
    \end{aligned}
    $$
Using \eqref{NLS-bdd} and \eqref{Lie-bdd}, we get
    \begin{equation}\label{A1_bdd}
    \| \mathcal{A}_1 \|_{\ell^{q}(I_0; L^r)}
    \leq C_{d,p} T_{loc}^{1-\frac{dp}{4}} \|\phi\|_{L^2(\mathbb{R}^d)}^p
    \big\| Z_{flt}(n\tau)\phi - P(\tau)u(n\tau) \big\|_{\ell^{q_0} (I_0; L^{r_0})} .
    \end{equation}

Second, to estimate $\mathcal{A}_2$ and $\mathcal{A}_5$, we use \eqref{MVT-1} and the Strichartz estimate \eqref{Str_inhomo2} again to get
    $$
    \begin{aligned}
    \| \mathcal{A}_2 \|_{\ell^{q}(I_0; L^r)}
    \leq & C_{d}  \Big\| \big| P(\tau)u(n\tau) - P(\widetilde{\tau})u(n\tau) \big| \Big( \big|P(\tau)u(n\tau)\big|^p + \big|P(\widetilde{\tau})u(n\tau)\big|^p \Big) \Big\|_{\ell^{q_0'} (I_0; L^{r_0'} )} \\
    \leq & C_{d} T_{loc}^{1-\frac{dp}{4}} \big\| P(\tau)u(n\tau) - P(\widetilde{\tau})u(n\tau) \big\|_{\ell^{q_0} (I_0; L^{r_0})} \\
    & \times \Big(  \big\| P(\tau)u(n\tau) \big\|_{\ell^{q_0} (I_0; L^{r_0})}^p + \big\| P(\widetilde{\tau})u(n\tau) \big\|_{\ell^{q_0} (I_0; L^{r_0})}^p \Big).
    \end{aligned}
    $$
Using \eqref{multiplier_3}
\footnote{
Since $\tau\leq \widetilde{\tau}$ and \eqref{multiplier_3}, we have a trivial bound
    $$
    \| P(\widetilde{\tau})u(n\tau) \|_{L^{r_0}(\mathbb{R}^d)}
    =\| P(\widetilde{\tau})P(\tau)u(n\tau) \|_{L^{r_0}(\mathbb{R}^d)}
    \leq C\| P(\tau)u(n\tau) \|_{L^{r_0}(\mathbb{R}^d)}.
    $$
}
and \eqref{NLS-bdd}, we get
    \begin{equation}\label{A2_bdd}
    \begin{split}
    \| \mathcal{A}_2 \|_{\ell^{q}(I_0; L^r)}
    \leq C_{d,p} T_{loc}^{1-\frac{dp}{4}} \|\phi\|_{L^2(\mathbb{R}^d)}^p
    \big\| P(\tau) u(n\tau) - P(\widetilde{\tau})u(n\tau) \big\|_{\ell^{q_0} (I_0; L^{r_0})}.
    \end{split}
    \end{equation}
Similarly, from the Strichartz estimate \eqref{Str_inhomo1}, \eqref{multiplier_3}, and \eqref{NLS-bdd}, we have
    \begin{equation}\label{A5_bdd}
    \begin{split}
    \| \mathcal{A}_5 \|_{\ell^{q}(I_0; L^r)}
    \leq C_{d,p} T_{loc}^{1-\frac{dp}{4}} \|\phi\|_{L^2(\mathbb{R}^d)}^p
    \big\| u - P(\widetilde{\tau})u \big\|_{L^{q_0} (I_0; L^{r_0})}.
    \end{split}
    \end{equation}

Next, to estimate $\mathcal{A}_3$, we use the Strichartz estimate \eqref{Str_inhomo2} and Mean Value Theorem \eqref{MVT-2} that yield
    \begin{equation}\label{A3_bdd2}
    \begin{aligned}
    \| \mathcal{A}_3 \|_{\ell^{q}(I_0; L^r)}
    &\leq C_d \bigg\| \frac{N(\tau) - I}{\tau} P(\widetilde{\tau}) u(n\tau) -i\lambda | P(\widetilde{\tau}) u(n\tau) |^p P(\widetilde{\tau})u(n\tau) \bigg\|_{\ell^{1} (I_0; L^2)} \\
    &\leq C_d\tau \big\| | P(\widetilde{\tau}) u(n\tau)|^{2p+1} \big\|_{\ell^{1} (I_0; L^2)}.
    \end{aligned}
    \end{equation}
Now, consider the admissible pair $(q_1,r_1)$, given by
    \begin{equation}\label{admisible_q1r1}
    \frac{1}{r_1} = \frac{p+1}{2(2p+1)}
    \quad\mbox{and}\quad
    \frac{1}{q_1} = \frac{dp}{4(2p+1)}.
    \end{equation}
We remark that the pair $(q_1,r_1)$ satisfies the admissible condition \eqref{admissible}.
From \eqref{multiplier_1}, \eqref{multiplier_3} and \eqref{NLS-bdd}, we have
    \begin{equation}\label{A3_bdd}
    \begin{aligned}
    \| \mathcal{A}_3 \|_{\ell^{q}(I_0; L^r)}
    &\leq C_d \tau \big\| P(\widetilde{\tau}) u(n\tau) \big\|_{\ell^{2p+1} (I_0; L^{2(2p+1)})}^{2p+1} \\
    &\leq C_d \tau T_{loc}^{(\frac{1}{2p+1}-\frac{1}{q_1})(2p+1)} \widetilde{\tau}^{\frac{d}{2}(\frac{1}{2(2p+1)}-\frac{1}{r_1}) (2p+1)} \big\| P(\widetilde{\tau}) u(n\tau) \big\|_{\ell^{q_1} (I_0; L^{r_1})}^{2p+1} \\
    &\leq C_d T_{loc}^{\frac{4-dp}{4} } \big(\tau/\widetilde\tau^{\frac{dp}{4}}\big) \big\| P(\tau) u(n\tau) \big\|_{\ell^{q_1} (I_0; L^{r_1})}^{2p+1} \\
    &\leq C_{d,p} T_{loc}^{\frac{4-dp}{4} } \frac{\tau}{\widetilde\tau} \| \phi \|_{L^2(\mathbb{R}^d)}^{2p+1} .
    \end{aligned}
    \end{equation}

Lastly, we consider the estimate of $\mathcal{A}_4$.
For simplicity, let us write
    $$
    \Phi(\cdot,s) = |P(\widetilde{\tau})u(s)|^p P(\widetilde{\tau})u(s).
    $$
By the Fundamental Theorem of Calculus, $\mathcal{A}_4$ becomes
    $$
    \begin{aligned}
    \mathcal{A}_4
    &= \tau \sum_{k=0}^{n-1} S(n \tau -k \tau)P(\tau) \Phi(\cdot,k\tau) - \int_0^{n\tau} S(n\tau -s)P(\tau) \Phi(\cdot,s) ds \\
    &= \sum_{k=0}^{n-1} \int_{k\tau}^{(k+1)\tau} \Big( S(n \tau -k \tau)P(\tau) \Phi(\cdot,k\tau) - S(n\tau -s)P(\tau) \Phi(\cdot,s) \Big) ds \\
    &= \sum_{k=0}^{n-1} \int_{k\tau}^{(k+1)\tau} \int_{k\tau}^{s} \partial_t \big( S(n \tau -t)P(\tau) \Phi(\cdot,t) \big) dt ds.
    \end{aligned}
    $$
Since
    $$
    \begin{aligned}
    \partial_t \big( S(n \tau- t)P(\tau) \Phi(\cdot,t) \big)^{\wedge}(\xi)
    &= \partial_t \Big( e^{i(n \tau- t)|\xi|^2} \chi(\tau^{\frac{1}{2}}\xi) \widehat{\Phi}(\xi,t) \Big)\\
    &= -i|\xi|^2 e^{i(n \tau- t)|\xi|^2} \chi(\tau^{\frac{1}{2}}\xi) \widehat{\Phi}(\xi,t) + e^{i(n \tau- t)|\xi|^2} \chi(\tau^{\frac{1}{2}}\xi)
    \widehat{\partial_t \Phi}(\xi,t),
    \end{aligned}
    $$
we have the identity
    $$
    \partial_t \big( S(n \tau -t)P(\tau) \Phi(\cdot,t) \big)
    = S(n \tau- t)P(\tau) ( \partial_t -i\Delta ) \Phi(\cdot,t).
    $$
Now, using this identity and integrating with respect to $s$, we have
    $$
    \begin{aligned}
    \mathcal{A}_4
    &=-i \sum_{k=0}^{n-1} \int_{k\tau}^{(k+1)\tau} \int_{k\tau}^s \big( S(n \tau- t)P(\tau) ( i\partial_t +\Delta ) \Phi(\cdot,t) \big) dtds \\
    &=-i \sum_{k=0}^{n-1} \int_{k\tau}^{(k+1)\tau} \big((k+1)\tau-t \big) \big( S(n \tau- t)P(\tau) ( i\partial_t +\Delta ) \Phi(\cdot,t) \big) dt \\
    &=-i \int_{0}^{n\tau} \Big( \sum_{k=0}^{n-1} \mathbf{1}_{(k\tau,(k+1)\tau)}(t)
    \big((k+1)\tau-t \big) \Big) \big( S(n \tau- t)P(\tau) ( i\partial_t +\Delta) \Phi(\cdot,t) \big) dt .
    \end{aligned}
    $$
Letting
    $$
    \Psi(t)
    = \sum_{k=0}^{\infty} \mathbf{1}_{(k\tau,(k+1)\tau)}(t) \big((k+1)\tau-t \big),
    $$
we note that $|\Psi(t)|\leq \tau$ for all $0\leq t<\infty$.
Thus, the Strichartz estimate \eqref{Str_inhomo1} gives
    \begin{equation}\label{A4_bdd2}
    \begin{aligned}
    \big\| \mathcal{A}_4 \big\|_{\ell^{q}(I_0; L^r)}
    &= \Big\| \int_{0}^{n\tau} S(n \tau- t)P(\tau) \Psi(t) \big( i\partial_t +\Delta \big) P(\tau) \Phi(\cdot,t) dt \Big\|_{\ell^{q}(I_0; L^r)}\\
    &\leq C_d \big\| P(\tau) \Psi(t) \big( i\partial_t +\Delta \big) \Phi(\cdot,t) \big\|_{L^1 (I_0; L^2)} \\
    &\leq C_d \tau \big\| P(\tau) \big(i\partial_t +\Delta \big) \big( |P(\widetilde{\tau})u|^p P(\widetilde{\tau})u \big) \big\|_{L^1(I_0; L^2)}.
    \end{aligned}
    \end{equation}
Here, using the identity
    $$
    i\partial_t u = -\Delta u + \lambda |u|^p u,
    $$
a direct calculation gives
    $$
    \begin{aligned}
    &(i\partial_t +\Delta) \big( |P(\widetilde{\tau})u|^p P(\widetilde{\tau})u \big) \\
    &\quad= \big( p |P(\widetilde{\tau})u|^{p-1} P(\widetilde{\tau})u + |P(\widetilde{\tau})u|^p \big) \big(i\partial_t P(\widetilde{\tau})u \big) + \Delta \big( |P(\widetilde{\tau})u|^{p} P(\widetilde{\tau})u \big) \\
    &\quad= \big( p |P(\widetilde{\tau})u|^{p-1} P(\widetilde{\tau})u + |P(\widetilde{\tau})u|^p \big) \Big(-\Delta P(\widetilde{\tau})u + P(\widetilde{\tau}) \big( \lambda |u|^p u\big) \Big)
    +\Delta \big( |P(\widetilde{\tau})u|^{p} P(\widetilde{\tau})u \big) .
    \end{aligned}
    $$
Thus, we obtain
    \begin{equation}\label{A4_bdd3}
    \begin{aligned}
    &\big\| P(\tau) (i\partial_t +\Delta) \big( |P(\widetilde{\tau})u|^p P(\widetilde{\tau})u \big) \big\|_{L^1(I_0; L^2)} \\
    &\qquad\leq C_p \big\| |P(\widetilde{\tau})u|^p (\Delta P(\widetilde{\tau})u) \big\|_{L^1(I_0; L^2)}
        + C_p \big\| |P(\widetilde{\tau})u|^{p} P(\widetilde{\tau})(|u|^p u) \big\|_{L^1(I_0; L^2)} \\
    &\qquad\qquad + C_p \big\| \Delta P(\tau) \big( |P(\widetilde{\tau})u|^p P(\widetilde{\tau})u \big) \big\|_{L^1(I_0; L^2)}.
    \end{aligned}
    \end{equation}
For a further estimate of \eqref{A4_bdd3}, we will use a similar argument as in the estimate of $\mathcal{A}_3$.
From \eqref{multiplier_2} and \eqref{multiplier_3}, the first term in the right hand side of \eqref{A4_bdd3} is estimated as
    $$
    \begin{aligned}
    \big\| | P(\widetilde{\tau}) u|^{p} (\Delta P(\widetilde{\tau}) u) \big\|_{L^{1}(I_0; L^{2})}
    &\leq \big\| | P(\widetilde{\tau}) u|^{p} \big\|_{L^{\frac{p+1}{p}}(I_0; L^{\frac{2(p+1)}{p}})} \big\| \Delta P(\widetilde{\tau}) u \big\|_{L^{p+1}(I_0; L^{2(p+1)})} \\
    &\leq C_{d,p} \frac{1}{\widetilde{\tau}} \big\| P(\widetilde{\tau}) u \big\|_{L^{p+1}(I_0; L^{2(p+1)})}^{p+1} \\
    &\leq C_{d,p} \frac{1}{\widetilde{\tau}} T_{loc}^{\frac{4-dp}{4}} \| u \|_{L^{\frac{4(p+1)}{dp}}(I_0; L^{2(p+1)})}^{p+1} .
    \end{aligned}
    $$
Also, from \eqref{multiplier_1}, the second term in \eqref{A4_bdd3} is controlled as
    $$
    \begin{aligned}
    &\big\| |P(\widetilde{\tau}) u|^p P(\widetilde{\tau}) \big(|u|^p u \big) \big\|_{L^{1}(I_0; L^{2})} \\
    &\qquad\leq \big\| P(\widetilde{\tau}) u \big\|_{L^{2p+1}(I_0; L^{2(2p+1)})}^p \big\| P(\widetilde{\tau}) \big(|u|^p u \big) \big\|_{L^{\frac{2p+1}{p+1}}(I_0; L^{\frac{2(2p+1)}{p+1}})} \\
    &\qquad\leq C_{d,p} \widetilde{\tau}^{\frac{d}{2} (\frac{1}{2(2p+1)}-\frac{1}{r_1})p } \| u \|_{L^{2p+1}(I_0; L^{r_1})}^p \widetilde{\tau}^{\frac{d}{2} (\frac{1}{2(2p+1)} - \frac{1}{r_1} )(p+1) } \big\| |u|^p u \big\|_{L^{\frac{2p+1}{p+1}}(I_0; L^{\frac{r_1}{p+1}})} \\
    &\qquad\leq C_{d,p} \widetilde\tau^{-\frac{dp}{4}} T_{loc}^{\frac{4-dp}{4}} \| u\|_{L^{q_1}(I_0; L^{r_1})}^{2p+1} ,
    \end{aligned}
    $$
where the pair $(q_1,r_1)$ is as in \eqref{admisible_q1r1}.
Lastly, from \eqref{multiplier_2} and the fractional chain rule, the third term in \eqref{A4_bdd3} is controlled as
    $$
    \begin{aligned}
    \big\| \Delta P(\tau) \big( |P(\widetilde{\tau})u|^p P(\widetilde{\tau})u \big) \big\|_{L^{1}(I_0; L^{2})}
    &\leq C_{d} \tau^{-\frac{1}{2}} \big\| \nabla \big( |P(\widetilde{\tau})u|^p P(\widetilde{\tau})u \big) \big\|_{L^{1}(I_0; L^{2})} \\
    &\leq C_{d,p} \tau^{-\frac{1}{2}} \big\| |P(\widetilde{\tau})u|^p \big( \nabla P(\widetilde{\tau})u \big)
    \big\|_{L^{1}(I_0; L^{2})} \\
    &\leq C_{d,p} \tau^{-\frac{1}{2}} \big\| | P(\widetilde{\tau}) u|^{p} \big\|_{L^{\frac{p+1}{p}}(I_0; L^{\frac{2(p+1)}{p}})} \big\| \nabla P(\widetilde{\tau}) u \big\|_{L^{p+1}(I_0; L^{2(p+1)})} \\
    &\leq C_{d,p} (\tau \widetilde{\tau})^{-\frac{1}{2}} \big\| P(\widetilde{\tau}) u \big\|_{L^{p+1}(I_0; L^{2(p+1)})}^{p+1} \\
    &\leq C_{d,p} (\tau \widetilde{\tau})^{-\frac{1}{2}} T_{loc}^{\frac{4-dp}{4}} \| u \|_{L^{\frac{4(p+1)}{dp}}(I_0; L^{2(p+1)})}^{p+1} .
    \end{aligned}
    $$
At this stage, let the constant $c_{d,p}>0$ in the definition of $T_{loc}$ be chosen so small that the local well-posedness \eqref{NLS-bdd} holds.
Since the pairs $(q_1,r_1)$ and $(\frac{4(p+1)}{dp},2(p+1))$ are admissible
\footnote{If we consider the higher dimensional case $d\geq 5$, this need an additional condition $p\leq \frac{2}{d-2}$.},
it now follows from \eqref{NLS-bdd} that
    \begin{equation}\label{A4_bdd}
    \begin{aligned}
    \| \mathcal{A}_4 \|_{\ell^{q}(I_0; L^r)}
    \leq C_{d,p} T_{loc}^{\frac{4-dp}{4}} \big( \| \phi \|_{L^2(\mathbb{R}^d)}^{p}
    + \| \phi \|_{L^2(\mathbb{R}^d)}^{2p} \big) \Big(\frac{\tau}{\widetilde{\tau}}\Big)^{\frac{1}{2}}  \| \phi \|_{L^2(\mathbb{R}^d)}.
    \end{aligned}
    \end{equation}

In all, applying \eqref{A1_bdd}, \eqref{A2_bdd}, \eqref{A5_bdd}, \eqref{A3_bdd}, and \eqref{A4_bdd} to \eqref{triangle_bdd}, we obtain
    \begin{equation}\label{key_estimate2}
    \begin{aligned}
    &\big\| Z_{flt}(n\tau)\phi  - P(\tau)u(n\tau) \big\|_{\ell^{q}(I_0; L^r)} \\
    &\quad\leq C_{d,p}  T_{loc}^{\frac{4-dp}{4}} \| \phi \|_{L^2(\mathbb{R}^d)}^{p}
    \big\| Z_{flt}(n\tau)\phi - P(\tau)u(n\tau) \big\|_{\ell^{q_0} (I_0; L^{r_0})} \\
    &\qquad + C_{d,p}  T_{loc}^{\frac{4-dp}{4}} \| \phi \|_{L^2(\mathbb{R}^d)}^{p}
    \bigg( \big\| P(\tau)u(n\tau) - P(\widetilde{\tau})u(n\tau) \big\|_{\ell^{q_0}(I_0; L^{r_0})}
    + \big\| u - P(\widetilde{\tau})u \big\|_{L^{q_0}(I_0; L^{r_0})} \bigg) \\
    &\qquad + C_{d,p} T_{loc}^{\frac{4-dp}{4}} \big( \| \phi \|_{L^2(\mathbb{R}^d)}^{p} + \| \phi \|_{L^2(\mathbb{R}^d)}^{2p} \big)
    \Big(\frac{\tau}{\widetilde{\tau}}\Big)^{\frac{1}{2}} \|\phi\|_{L^2(\mathbb{R}^d)}.
    \end{aligned}
    \end{equation}
Here, we choose the constant $c_{d,p} >0$ in the definition of $T_{loc}$ even so small that
    $$
    C_{d,p}  T_{loc}^{\frac{4-dp}{4}} \| \phi \|_{L^2(\mathbb{R}^d)}^{p}
    =C_{d,p}c_{d,p}^\frac{4-dp}{4}  \leq \frac{1}{4}.
    $$
Then the inequality \eqref{key_estimate2} with $(q,r)=(q_0,r_0)$ reduces to
    \begin{equation}\label{key_estimate3}
    \begin{aligned}
    &\big\| Z_{flt}(n\tau)\phi - P(\tau)u(n\tau) \big\|_{\ell^{q_0}(I_0; L^{r_0})} \\
    &\quad\leq \frac{1}{3} \Big( \big\| P(\tau)u(n\tau) - P(\widetilde{\tau})u(n\tau) \big\|_{\ell^{q_0}(I_0; L^{r_0})}
    + \big\| u - P(\widetilde{\tau})u \big\|_{L^{q_0}(I_0; L^{r_0})} \Big) \\
    &\qquad\quad + \frac{1}{3} \Big(\frac{\tau}{\widetilde{\tau}}\Big)^{\frac{1}{2}} \big( \|\phi\|_{L^2(\mathbb{R}^d)} +\|\phi\|_{L^2(\mathbb{R}^d)}^{p+1}\big).
    \end{aligned}
    \end{equation}
Inserting \eqref{key_estimate3} into \eqref{key_estimate2} again,
we finally obtain the desired estimate \eqref{key_estimate} for every admissible pair $(q,r)$.
\end{proof}

Lastly, we employ a standard argument to deduce the following property of the filtered Lie approximation $Z_{flt}$.

\begin{lem}\label{Z1-Z2_thm}
Let $d \geq 1$, $0< p < \frac{4}{d}$, and $\phi_1, \phi_2 \in L^2(\mathbb{R}^d)$.
Then there exists a constant $c_{d,p}>0$ such that, with $T_{loc}=c_{d,p} \min\big\{ \|\phi_1\|_{L^2(\mathbb{R}^d)}, \|\phi_2\|_{L^2(\mathbb{R}^d)} \big\}^{-\frac{4p}{4-dp}}$ and  $I_0= [0,T_{loc}]$,
    \begin{equation}\label{Z1-Z2_bdd}
    \big\| Z_{flt}(n\tau)\phi_1 - Z_{flt}(n\tau)\phi_2 \big\|_{\ell^{q}(I_0; L^{r}(\mathbb{R}^d))}
    \leq C_d \|\phi_1 - \phi_2\|_{L^2(\mathbb{R}^d)}
    \end{equation}
for any $0<\tau<1$ and admissible pair $(q,r)$.
\end{lem}

\begin{proof}
Let
    $$
    T_{loc}=c_{d,p} \min\big\{ \|\phi_1\|_{L^2(\mathbb{R}^d)}, \|\phi_2\|_{L^2(\mathbb{R}^d)} \big\}^{-\frac{4p}{4-dp}}
    \quad\mbox{and}\quad  I_0= [0,T_{loc}],
    $$
where $c_{d,p}>0$ is a constant to be determined below.
Let $(q_0,r_0)$ be the admissible pair, given by (\ref{admisible_q0r0}).
Using Duhamel's formula \eqref{Duhamel_Z} and the Strichartz estimates \eqref{Str_homo} and \eqref{Str_inhomo2}, we get
    \begin{equation}\label{Z1-Z2_bdd-2}
    \begin{aligned}
    &\big\| Z_{flt}(n\tau)\phi_1 - Z_{flt}(n\tau)\phi_2 \big\|_{\ell^{q}(I_0; L^{r})} \\
    &\quad\leq C_d \|\phi_1 - \phi_2\|_{L^2(\mathbb{R}^d)}
    + C_d \Big\| \frac{N(\tau) - I}{\tau} Z_{flt}(n\tau)\phi_1 - \frac{N(\tau) - I}{\tau} Z_{flt}(n\tau)\phi_2 \Big\|_{\ell^{q_0'}(I_0; L^{r_0'})}
    \end{aligned}
    \end{equation}
for any admissible pair $(q,r)$. Now, we choose the constant $c_{d,p}>0$ so small that Theorem \ref{thm_LWP} holds.
Then, from \eqref{MVT-1} and \eqref{Lie-bdd}, the second term in the right hand side of \eqref{Z1-Z2_bdd-2} is estimated as
    \begin{equation*}
    \begin{aligned}
    C_d \Big\|\frac{N(\tau) - I}{\tau} &  Z_{flt}(n\tau)\phi_1 - \frac{N(\tau) - I}{\tau} Z_{flt}(n\tau)\phi_2 \Big\|_{\ell^{q_0'}(I_0; L^{r_0'})} \\
    \leq &\, C_{d,p} \big\| \big( Z_{flt}(n\tau)\phi_1 - Z_{flt}(n\tau)\phi_2 \big) \big( |Z_{flt}(n\tau)\phi_1|^p + |Z_{flt}(n\tau)\phi_2|^p \big) \big\|_{\ell^{q_0'}(I_0; L^{r_0'})} \\
    \leq &\, C_{d,p} T_{loc}^{1-\frac{dp}{4}} \big\| Z_{flt}(n\tau)\phi_1 - Z_{flt}(n\tau)\phi_2 \big\|_{\ell^{q_0} (I_0; L^{r_0})} \\
    &\times \Big(  \big\| Z_{flt}(n\tau)\phi_1 \big\|_{\ell^{q_0} (I_0; L^{r_0})}^p + \big\| Z_{flt}(n\tau)\phi_2 \big\|_{\ell^{q_0} (I_0; L^{r_0})}^p \Big) \\
    \leq &\, C_{d,p} T_{loc}^{1-\frac{dp}{4}} \big\| Z_{flt}(n\tau)\phi_1 - Z_{flt}(n\tau)\phi_2 \big\|_{\ell^{q_0} (I_0; L^{r_0})} \Big(\|\phi_1\|^p_{L^2(\mathbb{R}^d)}+\|\phi_2\|^p_{L^2(\mathbb{R}^d)} \Big).
    \end{aligned}
    \end{equation*}
Again, we choose the constant $c_{d,p}>0$ even so small that
$$
C_{d,p}c_{d,p}^{1-\frac{dp}{4}}\le\frac{1}{4}.
$$
Then we have
    \begin{equation}\label{Z1-Z2_bdd-4}
    \begin{aligned}
    &\big\| Z_{flt}(n\tau)\phi_1 - Z_{flt}(n\tau)\phi_2 \big\|_{\ell^{q}(I_0; L^{r})} \\
    &\qquad\leq C_d \|\phi_1 - \phi_2\|_{L^2(\mathbb{R}^d)}
    + \frac{1}{2} \big\| Z_{flt}(n\tau)\phi_1 - Z_{flt}(n\tau)\phi_2 \big\|_{\ell^{q_0} (I_0; L^{r_0})}
    \end{aligned}
    \end{equation}
for any admissible pair $(q,r)$.

Taking $(q,r)=(q_0,r_0)$ in \eqref{Z1-Z2_bdd-4}, we arrive at \eqref{Z1-Z2_bdd} for the pair $(q,r)=(q_0,r_0)$.
Applying \eqref{Z1-Z2_bdd} for the pair $(q,r)=(q_0,r_0)$ to \eqref{Z1-Z2_bdd-4},
we obtain the desired inequality \eqref{Z1-Z2_bdd} for any admissible pair $(q,r)$.
\end{proof}


\section{Global-in-time error estimates}\label{sec_global}

In this section, we prove the main results of the paper, Theorem \ref{main-thm}.
From now on, instead of $u(t)$, we shall denote by $u(t)\phi$ the solution of \eqref{main-equation} at time $t$, corresponding to the initial datum $\phi$.

To start the proof, assume that $1\leq d\leq3$, $0< p < \frac{4}{d}$, and $\phi\in L^2(\mathbb{R}^d)$.
Fix any $T>0$ and $0<\tau<1$.
Also, let $c_{d,p}>0$ be the smallest one of the two constants $c_{d,p}$ from Proposition \ref{local_main_thm}, and Lemma \ref{Z1-Z2_thm}.

We consider only the meaningful case that $0<\tau <\min\big\{ \frac{c_{d,p}}{2} \|\phi\|_{L^2(\mathbb{R}^d)}^{-\frac{4p}{4-dp}}, 1 \big\}$.
In this case, there is a constant $\widetilde{c}_{d,p} \in (\frac{c_{d,p}}{2}, c_{d,p})$ such that
    $$
    \frac{\widetilde{c}_{d,p}\|\phi\|_{L^2(\mathbb{R}^d)}^{-\frac{4p}{4-dp}} }{\tau} \in \mathbb{N}.
    $$
Then we define
    $$
    T_{loc}= \widetilde{c}_{d,p} \|\phi\|_{L^2(\mathbb{R}^d)}^{-\frac{4p}{4-dp}}>0
    $$
and choose an integer $k_0\ge0$ such that $k_0 T_{loc}< T \leq (k_0 +1) T_{loc}$.
If $T \leq T_{loc}$, Theorem \ref{main-thm} is trivial from the local-in-time error estimates. So we assume $T>T_{loc}$, that is, $k_0>0$.
Note that the interval $[0,T]$ is covered by the intervals $I_k := [kT_{loc}, (k+1)T_{loc}]$ ($k=0,1,\cdots,k_0$).
Denote $[0,\widetilde{T}] =\cup_{k=0}^{k_0} I_k$;
then $[0,T] \subseteq [0,\widetilde{T}]$.

To show \eqref{main-result}, we start with the following inequality
    \begin{equation}\label{global_start}
    \begin{aligned}
    &\max_{0\leq n\tau\leq T} \big\| Z_{flt}(n \tau)\phi - u(n\tau)\phi \big\|_{L^2(\mathbb{R}^d)} \\
    &\quad\leq \big\| u(n \tau)\phi - P(\tau)u(n\tau)\phi \big\|_{\ell^{\infty}([0,T]; L^2)}
    + \big\| Z_{flt}(n \tau)\phi - P(\tau)u(n\tau)\phi \big\|_{\ell^{\infty}([0,T]; L^2)} \\
    &\quad\leq \big\| u(t)\phi - P(\tau)u(t)\phi \big\|_{L^{\infty}([0,T]; L^2)}
    + \max_{k=0,\cdots,k_0} \big\| Z_{flt}(n \tau)\phi - P(\tau)u(n\tau)\phi \big\|_{\ell^{\infty}(I_k; L^2)} .
    \end{aligned}
    \end{equation}
For $k=1,\ldots,k_0$, let
    $$
    n_k=\frac{kT_{loc}}{\tau}\in\mathbb{N}
    $$
so that $n_k\tau=kT_{loc}$ is the right and left endpoints of $I_{k-1}$ and $I_k$, respectively.
Using the identities
\[
    u(m_1 \tau+m_2 \tau)\phi = u(m_1 \tau)u(m_2 \tau)\phi
    \quad\mbox{and}\quad
    Z_{flt}(m_1 \tau +m_2 \tau)\phi = Z_{flt}(m_1 \tau) Z_{flt}(m_2 \tau)\phi
\]
for all integers $m_1, m_2\ge 0$,
the $k$-th part of the second term in the right hand side of \eqref{global_start} is bounded by
    \begin{equation}\label{final_bdd}
    \begin{aligned}
    &\big\|Z_{flt}(n \tau)\phi - P(\tau)u(n \tau)\phi \big\|_{\ell^{\infty}(I_k; L^2)} \\
    &\quad\leq \big\| Z_{flt}(n \tau) Z_{flt}(n_k \tau)\phi - P(\tau)u(n \tau) u(n_k \tau)\phi \big\|_{\ell^{\infty}(I_0; L^2)} \\
    &\quad\leq \big\| Z_{flt}(n \tau) Z_{flt}(n_k \tau)\phi - Z_{flt}(n \tau) P(\tau)u(n_k \tau)\phi \big\|_{\ell^{\infty}(I_0; L^2)} \\
    &\qquad\qquad + \big\| Z_{flt}(n \tau) P(\tau)u(n_k \tau)\phi - P(\tau)u(n \tau) u(n_k \tau)\phi \big\|_{\ell^{\infty}(I_0; L^2)}
    \end{aligned}
    \end{equation}
for any integer $1\leq k \leq k_0$.

First, we consider the first term in the right hand side of \eqref{final_bdd} to which we shall apply Lemma \ref{Z1-Z2_thm} with $\phi_1 = Z_{flt}(n_k \tau)\phi$ and $\phi_2 = P(\tau)u(n_k \tau)\phi$.
Observe from the \textit{mass conservation law} and the frequency localization in $L^2$ that
    $$
    \| \phi_1 \|_{L^2(\mathbb{R}^d)}\leq \| \phi \|_{L^2(\mathbb{R}^d)}
    \quad\mbox{and}\quad
    \| \phi_2 \|_{L^2(\mathbb{R}^d)}\leq \| \phi \|_{L^2(\mathbb{R}^d)};
    $$
thus, we can apply Lemma \ref{Z1-Z2_thm} to the first term of \eqref{final_bdd} to get
    $$
    \begin{aligned}
    \big\| Z_{flt}(n \tau) Z_{flt}(n_k \tau)\phi - Z_{flt}(n \tau) P(\tau)u(n_k \tau)\phi \big\|_{\ell^{\infty}(I_0; L^2)}
    &\leq C_d \big\| Z_{flt}(n_k \tau)\phi - P(\tau)u(n_k \tau)\phi \big\|_{L^2(\mathbb{R}^d)} \\
    &\leq C_d \big\| Z_{flt}(n\tau)\phi - P(\tau)u(n\tau)\phi \big\|_{\ell^{\infty}(I_{k-1}; L^2)},
    \end{aligned}
    $$
and then \eqref{final_bdd} becomes
    \begin{equation}\label{final_bdd2}
    \begin{aligned}
    &\big\|Z_{flt}(n \tau)\phi - P(\tau)u(n \tau)\phi \big\|_{\ell^{\infty}(I_k; L^2)} \\
    &\quad \leq C_d \big\| Z_{flt}(n\tau)\phi - P(\tau)u(n\tau)\phi \big\|_{\ell^{\infty}(I_{k-1}; L^2)}
    + \Big\| Z_{flt}(n \tau) u(n_k \tau)\phi - P(\tau)u(n \tau) u(n_k \tau)\phi \Big\|_{\ell^{\infty}(I_0; L^2)}
    \end{aligned}
    \end{equation}
for any integer $1\leq k\leq k_0$.

We are now ready to apply standard induction arguments to \eqref{final_bdd2}.
For $k=1,\ldots,k_0,$ using Proposition \ref{local_main_thm} with the initial datum $\phi_k =u(n_k \tau)\phi$,
the second term in the right hand side of \eqref{final_bdd2} is bounded by
    \begin{equation}\label{final_bdd3}
    \begin{aligned}
    &\big\| Z_{flt} (n \tau) u(n_k \tau)\phi - P(\tau)u(n \tau) u(n_k \tau)\phi \big\|_{\ell^{\infty}(I_0; L^2)} \\
    &\quad \leq \big\| P(\tau)u(n \tau) \phi_k - P(\widetilde{\tau}) u(n\tau) \phi_k \big\|_{\ell^{q_0}(I_0; L^{r_0})}
    + \big\| u(t) \phi_k - P(\widetilde{\tau}) u(t) \phi_k \big\|_{L^{q_0}(I_0; L^{r_0})} \\
    &\qquad\quad + \Big(\frac{\tau}{\widetilde{\tau}}\Big)^{\frac{1}{2}}  \big( \|\phi_k \|_{L^2(\mathbb{R}^d)} +\|\phi_k \|_{L^2(\mathbb{R}^d)}^{p+1}\big)
    \end{aligned}
    \end{equation}
for any $0<\tau <\min\big\{ \frac{c_{d,p}}{2} \|\phi\|_{L^2(\mathbb{R}^d)}^{-\frac{4p}{4-dp}}, 1 \big\}$ and $\tau\leq \widetilde\tau<1$.
Also, by the Strichartz estimates \eqref{Str_homo}, \eqref{Str_inhomo1} and H\"older's inequality, the first term of the right hand side of \eqref{final_bdd3} is bounded by
    \begin{equation}\label{final_bdd4}
    \begin{aligned}
    &\big\| P(\tau)u(n \tau) \phi_k - P(\widetilde{\tau}) u(n\tau) \phi_k \big\|_{\ell^{q_0}(I_0; L^{r_0})} \\
    &\leq C_d \big\| \big( P(\tau) - P(\widetilde{\tau}) \big) \phi_k \big\|_{L^{2}(\mathbb{R}^d)}
    + C_d \big\| \big( P(\tau) - P(\widetilde{\tau}) \big) \big( |u(t) \phi_k|^p u(t) \phi_k \big) \big\|_{L^{q_0'}(I_0; L^{r_0'})} \\
    &\leq C_d \big\| \big( P(\tau) - P(\widetilde{\tau}) \big) u(t)\phi \big\|_{L^{\infty}([0,T]; L^{2})}
    + C_d T_{loc}^{1-\frac{dp}{4}}  \big\| \big( P(\tau) - P(\widetilde{\tau}) \big) \big( |u(t) \phi|^p u(t) \phi \big) \big\|_{L^{\frac{q_0}{p+1}}(I_k; L^{\frac{r_0}{p+1}})} .
    \end{aligned}
    \end{equation}
Similar as \eqref{final_bdd4}, the second term of the right hand side of \eqref{final_bdd3} is bounded by
    \begin{equation}\label{final_bdd5}
    \begin{aligned}
    &\big\| u(t) \phi_k - P(\widetilde{\tau}) u(t) \phi_k \big\|_{L^{q_0}(I_0; L^{r_0})} \\
    &\quad\leq C_d \big\| \big( 1 - P(\widetilde{\tau}) \big) u(t)\phi \big\|_{L^{\infty}([0,T]; L^{2})}
    + C_d T_{loc}^{1-\frac{dp}{4}}  \big\| \big( 1 - P(\widetilde{\tau}) \big) \big( |u(t) \phi|^p u(t) \phi \big) \big\|_{L^{\frac{q_0}{p+1}}(I_k; L^{\frac{r_0}{p+1}})} .
    \end{aligned}
    \end{equation}
Applying \eqref{final_bdd3}, \eqref{final_bdd4} and \eqref{final_bdd5} to \eqref{final_bdd2},
we have
    \begin{equation}\label{final_bdd6}
    \begin{aligned}
    \big\|Z_{flt}(n \tau)\phi - P(\tau)u(n \tau)\phi \big\|_{\ell^{\infty}(I_k; L^2)}
    \leq C_d \big\| Z_{flt}(n\tau)\phi - P(\tau)u(n\tau)\phi \big\|_{\ell^{\infty}(I_{k-1}; L^2)}
    + \mathcal{B}_1
    \end{aligned}
    \end{equation}
for any integer $1\leq k\leq k_0$, where
    $$
    \begin{aligned}
    \mathcal{B}_1
    &:= C_d \Big( \big\| \big( 1 - P(\widetilde{\tau}) \big) u(t)\phi \big\|_{L^{\infty}([0,T]; L^{2})}
    + \big\| \big( 1 - P(\tau) \big) u(t)\phi \big\|_{L^{\infty}([0,T]; L^{2})} \Big) \\
    &\qquad + C_d T_{loc}^{1-\frac{dp}{4}} \bigg( \big\| \big( 1 - P(\widetilde{\tau}) \big) \big( |u(t) \phi|^p u(t) \phi \big) \big\|_{L^{\frac{q_0}{p+1}}([0,\widetilde{T}]; L^{\frac{r_0}{p+1}})} \\
    &\qquad\qquad\quad + \big\| \big( 1 - P(\tau) \big) \big( |u(t) \phi|^p u(t) \phi \big) \big\|_{L^{\frac{q_0}{p+1}}([0,\widetilde{T}]; L^{\frac{r_0}{p+1}})} \bigg)
    + \Big(\frac{\tau}{\widetilde{\tau}}\Big)^{\frac{1}{2}} \big( \|\phi\|_{L^2(\mathbb{R}^d)} +\|\phi\|_{L^2(\mathbb{R}^d)}^{p+1}\big).
    \end{aligned}
    $$
Note also that $\big\| Z_{flt}(n \tau)\phi - P(\tau)u(n \tau)\phi \big\|_{\ell^{\infty}(I_0; L^2)}\le\mathcal{B}_1$.
Apply \eqref{final_bdd6} to the second term in the right hand side of \eqref{global_start},
we have
    $$
    \begin{aligned}
    &\max_{k=0,\cdots,k_0} \big\| Z_{flt}(n \tau)\phi - P(\tau)u(n \tau)\phi \big\|_{\ell^{\infty}(I_k; L^2)} \leq \sum_{k=0}^{k_0}(1+C_d)^k \mathcal{B}_1 \\
    &\qquad\qquad\qquad \leq \frac{(1+C_d)^{\frac{T}{T_{loc}} +1}}{C_d} \mathcal{B}_1
    =\frac{1+C_d}{C_d} \exp{\Big(\frac{\log (1+C_d)}{\widetilde{c}_{d,p}}T\|\phi\|^{\frac{4p}{4-dp}}_{L^2(\mathbb{R}^d)}\Big)} \mathcal{B}_1.
    \end{aligned}
    $$
Thus, the convergence estimate over the global time interval $[0,T]$ reduces to
    \begin{equation}\label{final_bdd7}
    \begin{aligned}
    &\max_{0\leq n\tau\leq T} \big\| Z_{flt}(n \tau)\phi - u(n\tau)\phi \big\|_{L^2(\mathbb{R}^d)} \\
    &\qquad\leq \big\| u(t)\phi - P(\tau)u(t)\phi \big\|_{L^{\infty}([0,T]; L^2)}
    + \frac{1+C_d}{C_d} \exp\big( C_{d,p}T \| \phi\|_{L^2(\mathbb{R}^d)}^{\frac{4p}{4-dp}} \big) \mathcal{B}_1 \\
    &\qquad\leq \frac{1+ 2C_d}{C_d} \exp\big( C_{d,p}T \| \phi\|_{L^2(\mathbb{R}^d)}^{\frac{4p}{4-dp}} \big) \mathcal{B}_1
    \end{aligned}
    \end{equation}
for all $0<\tau <\min\big\{ \frac{c_{d,p}}{2} \|\phi\|_{L^2(\mathbb{R}^d)}^{-\frac{4p}{4-dp}}, 1 \big\}$ and $\tau\leq \widetilde\tau<1$.

Since $\phi\in L^2(\mathbb{R}^d)$, the local well-posedness theory \eqref{NLS-bdd} gives the global well-posedness result
    $$
    u(t)\phi \in L^{\infty}([0,T]; L^{2})
    \quad\mbox{and}\quad
    |u(t)\phi|^p u(t)\phi \in L^{\frac{q_0}{p+1}}([0,\widetilde{T}]; L^{\frac{r_0}{p+1}}).
    $$
More preciesely, for the initial data $\phi_k=u(n_k\tau)\phi$ and the admissible pair $(q,r)$, Theorem \ref{thm_LWP} and the \textit{mass conservation law} give
    \begin{equation}\label{NLS_global_bdd}
    \begin{aligned}
    \| u(t)\phi \|_{L^{q}([0,\widetilde{T}]; L^{r})}
    \leq \sum_{k=0}^{k_0} \| u(t)\phi_k \|_{L^{q}(I_0; L^{r})}
    \leq \frac{\widetilde{T}}{T_{loc}} C_{d} \|\phi\|_{L^2(\mathbb{R}^d)}
    \leq C_{d,p} T \|\phi\|_{L^2(\mathbb{R}^d)}^{C_{d,p}} .
    \end{aligned}
    \end{equation}
Also, $|u(t)\phi|^p u(t)\phi \in L^{\frac{q_0}{p+1}}([0,\widetilde{T}]; L^{\frac{r_0}{p+1}})$ follows from H\"older's inequality.

Thus, by the Dominated Convergence Theorem, we have
    $$
    \big\| \big(1- P(\tau)\big) u(t)\phi \big\|_{L^{\infty}([0,T]; L^2)}
    + \big\| \big( 1 - P(\tau) \big) \big( |u(t) \phi|^p u(t) \phi \big) \big\|_{L^{\frac{q_0}{p+1}}([0,\widetilde{T}]; L^{\frac{r_0}{p+1}})}
    \rightarrow 0
    \quad\mbox{as}\quad \tau\rightarrow 0 ,
    $$
and
    $$
    \big\| \big(1- P(\widetilde{\tau})\big) u(t)\phi \big\|_{L^{\infty}([0,T]; L^2)}
    + \big\| \big( 1 - P(\widetilde{\tau}) \big) \big( |u(t) \phi|^p u(t) \phi \big) \big\|_{L^{\frac{q_0}{p+1}}([0,\widetilde{T}]; L^{\frac{r_0}{p+1}})}
    \rightarrow 0
    \quad\mbox{as}\quad \widetilde{\tau}\rightarrow 0 .
    $$
Putting $\widetilde{\tau}=\tau^{\frac{1}{2}}$
and taking $\tau \rightarrow 0$ in \eqref{final_bdd7}, then $\mathcal{B}_1$ converges to zero.
That means, we obtain the convergence \eqref{main-result}.


\section{Proof of Theorem \ref{main-thm3}}\label{sec_radial}

In this section, we establish the convergence estimate (\ref{main-result5})  for $L^2$ radial initial data.
In the first subsection, we begin with some preparations. Then we give a proof of Theorem \ref{main-thm3} in the second subsection.

\subsection{Preliminaries}
Let $\widetilde{P}$ be the frequency localized operator, given by
    $$
    \widetilde{P}(\widetilde{\tau}) = P(\widetilde{\tau}) -P(4\widetilde{\tau}),
    $$
where $P$ is the frequency localized multiplier in (\ref{freq-loc-oper}).
Then the Fourier transform of $\widetilde{P}(\widetilde{\tau})\phi$ is supported on $B(0,\widetilde{\tau}^{-\frac{1}{2}})\setminus B(0,\frac{1}{4}\widetilde{\tau}^{-\frac{1}{2}})$,
and
    \begin{equation}\label{part_unit}
    \sum_{k\in\mathbb{Z}} \widetilde{P}(4^{-k} \widetilde{\tau}) \phi(x) = \phi(x).
    \end{equation}

\begin{theoremalpha}[Strichartz estimates for radial initial data]\label{thm-str-rad}
Let $d\geq 2$.
Then there exists a constant $C_{d}>0$ such that
    \begin{equation}\label{Str_homo-rad}
    \big\| S(t) \widetilde{P}(1) \phi \big\|_{L^q (\mathbb{R}; L^r (\mathbb{R}^d))}
    \leq C_{d} \|\phi\|_{L^2 (\mathbb{R}^d)}
    \end{equation}
for all pairs $(q,r)$ with
    \begin{equation}\label{radial_pair}
    \frac{2}{q} +\frac{2d-1}{r} \leq \frac{2d-1}{2},
    \quad q\geq 2,\quad\mbox{and}\quad(q,r)\neq \Big(2,\frac{4d-2}{2d-3} \Big),
    \end{equation}
    and radial functions $\phi\in L^2 (\mathbb{R}^d)$.
\end{theoremalpha}

There have been many studies on the Strichartz estimates for radial initial data.
For the history and some applications to nonlinear equations, we refer to Guo-Wang \cite{GW} and references therein. Also, for the simplest proof of \eqref{Str_homo-rad}, we refer to Ke \cite{K}.
Lastly, we point out that the endpoint estimate \eqref{Str_homo-rad} for $(q,r)=(2,\frac{4d-2}{2d-3})$ is still an open question.

Let $d\geq 2$ and $0<\tau \leq \widetilde{\tau}<1$. The range of pairs $(q,r)$ satisfying \eqref{radial_pair} is wider than that of admissible pairs $(q,r)$ fulfilling \eqref{admissible}.
This extended range will play a key role in proving Theorem \ref{main-thm3}.

Upon scaling, \eqref{Str_homo-rad} can rewritten as
    \begin{equation}\label{Str_homo-rad2}
    \big\| S(t) \widetilde{P}(\widetilde{\tau}) \phi \big\|_{L^q (\mathbb{R}; L^r (\mathbb{R}^d))}
    \leq C_{d} \widetilde{\tau}^{-\frac{1}{2} (\frac{d}{2}-\frac{d}{r}-\frac{2}{q})} \|\phi\|_{L^2 (\mathbb{R}^d)}
    \end{equation}
for all $(q,r)$ satisfying \eqref{radial_pair} and radial functions $\phi\in L^2 (\mathbb{R}^d)$.

Suppose  $1\le p<\frac{4}{d}$, and consider the admissible pair $(q_0,r_0)$, defined by \eqref{admisible_q0r0}.
Let $q_2$ be given by
    $$
    \frac{1}{q_2} = \frac{1}{q_0} +\epsilon_0,
    $$
where
    $$
    \epsilon_0 := \frac{4-dp}{8} >0.
    $$
Then we can easily check that $(q_2,r_0)$ satisfies \eqref{radial_pair}.

From the $TT^*$ argument and Christ-Kiselev lemma, \eqref{Str_homo} and \eqref{Str_homo-rad2} imply
    $$
    \begin{aligned}
    &\Big\| \int_{s < n \tau} S(n \tau -s) \widetilde{P}(\widetilde{\tau}) F(s) ds \Big\|_{\ell^{q_0} (\mathbb{R}; L^{r_0} )}
    + \Big\| \int_{s< t} S(t-s) \widetilde{P}(\widetilde{\tau}) F(s) ds \Big\|_{L^{q_0} (\mathbb{R}; L^{r_0})} \\
    &\qquad\qquad \leq C_{d} \widetilde{\tau}^{-\frac{1}{2} (\frac{d}{2}-\frac{d}{r_0}-\frac{2}{q_2})} \| F\|_{L^{q_2'} (\mathbb{R}; L^{r_0'} )}
    \end{aligned}
    $$
for all radial functions $F\in L^{q_2'}(\mathbb{R}; L^{r_0'}(\mathbb{R}^d))$ with respect to $x$ and $0<\tau\leq \widetilde{\tau}<1$.
Thus, from the triangle inequality and \eqref{part_unit}, we have
    \begin{equation}\label{Str_inhomo-rad2}
    \begin{aligned}
    &\Big\| \big(1-P(\widetilde{\tau})\big) \int_{s< t} S(t-s) F(s) ds \Big\|_{L^{q_0} (\mathbb{R}; L^{r_0})} \\
    &\qquad\leq \sum_{k=0}^{\infty} \Big\| \int_{s< t} S(t-s) \widetilde{P}(4^{-k}\widetilde{\tau}) F(s) ds \Big\|_{L^{q_0} (\mathbb{R}; L^{r_0})} \\
    &\qquad\leq \sum_{k=0}^{\infty} C_{d} (4^{-k}\widetilde{\tau})^{\epsilon_0} \| F\|_{L^{q_2'} (\mathbb{R}; L^{r_0'} )} \leq C_{d} \widetilde{\tau}^{\epsilon_0} \| F\|_{L^{q_2'} (\mathbb{R}; L^{r_0'} )}
    \end{aligned}
    \end{equation}
and
    \begin{equation}\label{Str_inhomo-rad2'}
    \begin{aligned}
    &\Big\| \big(P(\tau)-P(\widetilde{\tau})\big) \int_{s < n \tau} S(n \tau -s) F(s) ds \Big\|_{\ell^{q_0} (\mathbb{R}; L^{r_0})} \\
    &\qquad\leq \sum_{k=0}^{\log_4(\widetilde{\tau}/\tau)} \Big\| \int_{s < n \tau} S(n \tau -s) \widetilde{P}(4^{-k}\widetilde{\tau}) F(s) ds \Big\|_{\ell^{q_0} (\mathbb{R}; L^{r_0})} \\
    &\qquad\leq \sum_{k=0}^{\log_4(\widetilde{\tau}/\tau)} C_{d} (4^{-k}\widetilde{\tau})^{\epsilon_0} \| F\|_{L^{q_2'} (\mathbb{R}; L^{r_0'} )} \leq C_{d} \widetilde{\tau}^{\epsilon_0} \| F\|_{L^{q_2'} (\mathbb{R}; L^{r_0'} )}
    \end{aligned}
    \end{equation}
for all radial functions $F\in L^{q_2'}(\mathbb{R}; L^{r_0'}(\mathbb{R}^d))$ with respect to $x$.

\subsection{Proof of Theorem \ref{main-thm3}}
Let $2\leq d\leq 3$ and $0< p<\frac{4}{d}$. Assume that $\phi\in L^2(\mathbb{R}^d)$ is radially symmetric in $x$. Let $c_{d,p}>0$ be the constant from Section \ref{sec_global}, and we consider any $0<\tau<\min\big\{\frac{c_{d,p}}{2}\|\phi\|^{-\frac{4p}{4-dp}}_{L^2(\mathbb{R}^d)},1\big\}$ and $\tau \leq \widetilde{\tau} <1$.
    \footnote{
    The error estimate for the trivial case  $\tau\ge\min\big\{\frac{c_{d,p}}{2}\|\phi\|^{-\frac{4p}{4-dp}}_{L^2(\mathbb{R}^d)},1\big\}$ can be achieved from a similar argument with an alternative constant $\widetilde{c}_{d,p} \in (\frac{c_{d,p}}{4}, \frac{c_{d,p}}{2})$ such that $\tau / \big( \widetilde{c}_{d,p} \|\phi\|_{L^2(\mathbb{R}^d)}^{-\frac{4p}{4-dp}}\big) \in \mathbb{N}$ and $T_{loc} = \widetilde{c}_{d,p} \|\phi\|_{L^2(\mathbb{R}^d)}^{-\frac{4p}{4-dp}}$;
    we omit the details for this case.}

\underline{\textbf{Local estimate.}}
Let
    $$
    T_{loc}=\widetilde{c}_{d,p} \|\phi\|_{L^2(\mathbb{R}^d)}^{-\frac{4p}{4-dp}}\quad\mbox{and}\quad I_0= [0,T_{loc}],
    $$
where $\widetilde{c}_{d,p}>0$ is the constant from Section \ref{sec_global}.
Since the equation in \eqref{main-equation} is radially symmetric in $x$, so is the solution $u(t)\phi$ in $x$.
Thus, applying the Strichartz estimates \eqref{Str_homo}, \eqref{Str_inhomo-rad2} and \eqref{Str_inhomo-rad2'}, we have
    \begin{equation}\label{Str_inhomo-rad3}
    \begin{aligned}
    &\big\| u(t)\phi - P(\widetilde{\tau})u(t)\phi \big\|_{L^{q}(I_0; L^{r})} \\
    &\quad \leq \big\| S(t) (1- P(\widetilde{\tau}))\phi \big\|_{L^{q_0}(I_0; L^{r_0})}
        + \Big\| \big(1- P(\widetilde{\tau})\big) \int_{s < n \tau} S(n \tau -s) \big( |u(s)\phi |^p u(s)\phi \big) ds \Big\|_{L^{q_0} (I_0; L^{r_0})} \\
    &\quad \leq C_d \| \phi - P(\widetilde{\tau})\phi \|_{L^2(\mathbb{R}^d)}
    + C_d \widetilde{\tau}^{\epsilon_0} \big\| |u(t)\phi |^p u(t)\phi \big\|_{L^{q_2'} (I_0; L^{r_0'})} \\
    &\quad \leq C_d \| \phi - P(\widetilde{\tau})\phi \|_{L^2(\mathbb{R}^d)}
    + C_{d,p} \widetilde{\tau}^{\epsilon_0} T_{loc}^{\frac{1}{q_2'} - \frac{1}{q_0}} \| u(t)\phi \|_{L^{q_0} (I_0; L^{r_0})}^{p+1}
    \end{aligned}
    \end{equation}
and
    \begin{equation}\label{Str_inhomo-rad4}
    \begin{aligned}
    &\big\| P(\tau)u(n\tau)\phi - P(\widetilde{\tau})u(n\tau)\phi \big\|_{\ell^{q}(I_0; L^{r})} \\
    &\quad\leq C_{d} \| P(\tau)\phi - P(\widetilde{\tau})\phi \|_{L^2(\mathbb{R}^d)}
    + C_{d,p} \widetilde{\tau}^{\epsilon_0} T_{loc}^{\frac{1}{q_2'} - \frac{1}{q_0}} \big\| u(t)\phi \big\|_{L^{q_0} (I_0; L^{r_0})}^{p+1}
    \end{aligned}
    \end{equation}
for any admissible pair $(q,r)$.
From \eqref{NLS-bdd} and $T_{loc}=\widetilde{c}_{d,p} \|\phi\|_{L^2}^{-\frac{4p}{4-dp}}$, \eqref{Str_inhomo-rad3} and \eqref{Str_inhomo-rad4} give
    \begin{equation}\label{Str_inhomo-rad5}
    \begin{aligned}
    &\big\| u(t)\phi - P(\widetilde{\tau})u(t)\phi \big\|_{L^{q_0}(I_0; L^{r_0})}
    + \big\| P(\tau)u(n\tau)\phi - P(\widetilde{\tau})u(n\tau)\phi \big\|_{\ell^{q_0}(I_0; L^{r_0})} \\
    &\qquad \leq C_{d} \Big( \| \phi - P(\widetilde{\tau})\phi \|_{L^2(\mathbb{R}^d)} + \| \phi - P(\tau)\phi \|_{L^2(\mathbb{R}^d)} \Big)
    + C_{d,p} \widetilde{\tau}^{\frac{4-dp}{8}} T_{loc}^{\frac{4-dp}{8}} \| \phi \|_{L^2(\mathbb{R}^d)}^{p+1} \\
    &\qquad \leq 2C_{d} \| \phi - P(\widetilde{\tau})\phi \|_{L^2(\mathbb{R}^d)}
    + C_{d,p} \widetilde{\tau}^{\frac{4-dp}{8}} \| \phi \|_{L^2(\mathbb{R}^d)}^{\frac{p+2}{2}}.
    \end{aligned}
    \end{equation}
Now, applying \eqref{Str_inhomo-rad5} to Proposition \ref{local_main_thm}, we get the following local-in-time estimate:
    \begin{equation}\label{key_estimate_rad}
    \begin{aligned}
    &\big\| Z_{flt}(n\tau)\phi - P(\tau)u(n\tau)\phi \big\|_{\ell^{q}(I_0; L^r)} \\
    &\quad\leq 2C_{d} \| \phi - P(\widetilde{\tau})\phi \|_{L^2(\mathbb{R}^d)}
    + C_{d,p} \widetilde{\tau}^{\frac{4-dp}{8}} \| \phi \|_{L^2(\mathbb{R}^d)}^{\frac{p+2}{2}} + \Big(\frac{\tau}{\widetilde{\tau}}\Big)^{\frac{1}{2}} \big( \|\phi\|_{L^2(\mathbb{R}^d)} +\|\phi\|_{L^2(\mathbb{R}^d)}^{p+1}\big)
    \end{aligned}
    \end{equation}
for any $0<\tau<\min\big\{\frac{c_{d,p}}{2}\|\phi\|^{-\frac{4p}{4-dp}}_{L^2(\mathbb{R}^d)},1\big\}$, $\tau \leq \widetilde{\tau} <1$ and admissible pair $(q,r)$.

\
\underline{\textbf{Global estimate.}}
We now prove the global-in-time convergence estimate \eqref{main-result5}.
This proof is the same as that of Theorem \ref{main-thm}.
Fix any $T>0.$ Since the case that $T\le T_{loc}$ follows from the local-in-time estimate above, we assume $T> T_{loc}$. Let $k_0\in\mathbb{N}$ be such that $k_0 T_{loc}<T\le (k_0+1) T_{loc}$, and take $I_k := [kT_{loc}, (k+1)T_{loc}]$ ($k=0,1,\cdots,k_0$) as in Section \ref{sec_global}.
Applying \eqref{key_estimate_rad} to \eqref{final_bdd3},
then \eqref{final_bdd2} is bounded by
    \begin{equation}\label{final_bdd_rad}
    \begin{aligned}
    &\big\|Z_{flt}(n \tau) \phi - P(\tau)u(n \tau)\phi \big\|_{\ell^{\infty}(I_k; L^2)} \\
    &\quad\leq C_d \big\| Z_{flt}(n\tau)\phi - P(\tau)u(n\tau)\phi \big\|_{\ell^{\infty}(I_{k-1}; L^2)} \\
    &\qquad\quad + C_{d} \big\| u(n_k\tau)\phi - P(\widetilde{\tau})u(n_k\tau)\phi \big\|_{L^2(\mathbb{R}^d)} \\
    &\qquad\quad + C_{d,p}\widetilde{\tau}^{\frac{4-dp}{8}} \| u(n_k\tau)\phi \|_{L^2(\mathbb{R}^d)}^{\frac{p+2}{2}} + \Big(\frac{\tau}{\widetilde{\tau}}\Big)^{\frac{1}{2}} \big( \| u(n_k\tau)\phi\|_{L^2(\mathbb{R}^d)} +\| u(n_k\tau)\phi\|_{L^2(\mathbb{R}^d)}^{p+1}\big)  \\
    &\quad\leq C_d \big\| Z_{flt}(n\tau)\phi - P(\tau)u(n\tau)\phi \big\|_{\ell^{\infty}(I_{k-1}; L^2)}
    + \mathcal{B}_2
    \end{aligned}
    \end{equation}
for any $0\le k\le k_0$, $0<\tau<\min\big\{\frac{c_{d,p}}{2}\|\phi\|^{-\frac{4p}{4-dp}}_{L^2(\mathbb{R}^d)},1\big\}$ and $\tau \leq \widetilde{\tau} <1$, where $I_{-1}:=\emptyset$.
In the above, we define $\mathcal{B}_2$ as
    $$
    \begin{aligned}
    \mathcal{B}_2
    := C_d \big\| \big( 1 - P(\widetilde{\tau}) \big) u(t)\phi \big\|_{L^{\infty}([0,T]; L^{2})}
    + C_{d,p}\widetilde{\tau}^{\frac{4-dp}{8}} \| \phi \|_{L^2(\mathbb{R}^d)}^{\frac{p+2}{2}} + \Big(\frac{\tau}{\widetilde{\tau}}\Big)^{\frac{1}{2}} \big( \|\phi\|_{L^2(\mathbb{R}^d)} +\|\phi\|_{L^2(\mathbb{R}^d)}^{p+1}\big).
    \end{aligned}
    $$
Using the induction argument with \eqref{final_bdd_rad},
we have
    $$
    \begin{aligned}
    \max_{k=0,\cdots,k_0} \big\| Z_{flt}(n \tau)\phi - P(\tau)u(n\tau)\phi \big\|_{\ell^{\infty}(I_k; L^2)}
    \leq \sum_{k=0}^{k_0} (1+C_d)^k \mathcal{B}_2
    \leq C_d \exp\big( C_{d,p}T \| \phi\|_{L^2(\mathbb{R}^d)}^{\frac{4p}{4-dp}} \big) \mathcal{B}_2 .
    \end{aligned}
    $$
Applying this inequality to \eqref{global_start}, we have
\footnote{From $0<\tau \leq \widetilde{\tau}$, we get
$\| (1-P(\tau))u(t)\phi \big\|_{L^2(\mathbb{R}^d)} \leq \| (1-P(\widetilde{\tau}))u(t)\phi \big\|_{L^2(\mathbb{R}^d)}$.}
    \begin{equation}\label{global_start_rad}
    \begin{aligned}
    &\max_{0\leq n\tau\leq T} \big\| Z_{flt}(n \tau)\phi - u(n\tau)\phi \big\|_{L^2(\mathbb{R}^d)} \\
    &\qquad\leq \big\| u(t)\phi - P(\tau)u(t)\phi \big\|_{L^{\infty}([0,T]; L^2)}
    + C_d \exp\big( C_{d,p}T \| \phi\|_{L^2(\mathbb{R}^d)}^{\frac{4p}{4-dp}} \big) \mathcal{B}_2 \\
    &\qquad\leq (1+C_d) \exp\big( C_{d,p}T \| \phi\|_{L^2(\mathbb{R}^d)}^{\frac{4p}{4-dp}} \big) \mathcal{B}_2 .
    \end{aligned}
    \end{equation}
Furthermore, we can control the first term of $\mathcal{B}_2$ by the initial data $\phi$.
\eqref{Str_inhomo-rad3} gives
    \begin{equation}
    \begin{aligned}
    &\big\| u(t)\phi - P(\widetilde{\tau})u(t)\phi \big\|_{L^{\infty}(I_k; L^{2})}
    = \big\| u(t)u(n_k\tau)\phi - P(\widetilde{\tau})u(t)u(n_k\tau)\phi \big\|_{L^{\infty}(I_0; L^{2})} \\
    &\qquad\leq C_d \| u(n_k\tau)\phi - P(\widetilde{\tau})u(n_k\tau)\phi \|_{L^2(\mathbb{R}^d)}
    + C_{d,p} \widetilde{\tau}^{\frac{4-dp}{8}} T_{loc}^{\frac{4-dp}{8}} \| u(t)u(n_k\tau)\phi \|_{L^{q_0} (I_0; L^{r_0})}^{p+1} \\
    &\qquad\leq C_d \| u(t)\phi - P(\widetilde{\tau})u(t)\phi \|_{L^{\infty}(I_{k-1}; L^{2})}
    + C_{d,p} \widetilde{\tau}^{\frac{4-dp}{8}} \|\phi\|_{L^{2}(\mathbb{R}^d)}^{-\frac{p}{2}} \| u(t)\phi \|_{L^{q_0} ([0,\widetilde{T}]; L^{r_0})}^{p+1}
    \end{aligned}
    \end{equation}
and
    \begin{equation}
    \begin{aligned}
    \big\| u(t)\phi - P(\widetilde{\tau})u(t)\phi \big\|_{L^{\infty}(I_0; L^{2})}
    \leq C_d \| \phi - P(\widetilde{\tau})\phi \|_{L^{2}(\mathbb{R}^d)}
    + C_{d,p} \widetilde{\tau}^{\frac{4-dp}{8}} \|\phi\|_{L^{2}(\mathbb{R}^d)}^{-\frac{p}{2}} \| u(t)\phi \|_{L^{q_0} ([0,\widetilde{T}]; L^{r_0})}^{p+1}.
    \end{aligned}
    \end{equation}
This implies
    \begin{equation}\label{global_start_rad2}
    \begin{aligned}
    &\big\| u(t)\phi - P(\widetilde{\tau})u(t)\phi \big\|_{L^{\infty}([0,T]; L^{2})} \\
    &\qquad\leq \max_{k=0,\cdots,k_0} \big\| u(t)\phi - P(\widetilde{\tau})u(t)\phi \big\|_{L^{\infty}(I_k; L^{2})} \\
    &\qquad\leq (1+C_d)^{k_0} \Big( \| \phi - P(\widetilde{\tau})\phi \|_{L^{2}(\mathbb{R}^d)} + C_{d,p} \widetilde{\tau}^{\frac{4-dp}{8}} \|\phi\|_{L^{2}(\mathbb{R}^d)}^{-\frac{p}{2}} \| u(t)\phi \|_{L^{q_0} ([0,\widetilde{T}]; L^{r_0})}^{p+1} \Big) \\
    &\qquad\leq C_{d,p} \exp\big( C_{d,p}T \| \phi\|_{L^2(\mathbb{R}^d)}^{\frac{4p}{4-dp}} \big) \Big( \| \phi - P(\widetilde{\tau})\phi \|_{L^{2}(\mathbb{R}^d)} + \widetilde{\tau}^{\frac{4-dp}{8}} T\|\phi \|_{L^{2}(\mathbb{R}^d)}^{C_{d,p}} \Big)
    \end{aligned}
    \end{equation}
from the global well-posedness \eqref{NLS_global_bdd}.
Applying \eqref{global_start_rad2} to \eqref{global_start_rad},
we have
    \begin{equation}\label{global_start_rad3}
    \begin{aligned}
    &\max_{0\leq n\tau\leq T} \big\| Z_{flt}(n \tau)\phi - u(n\tau)\phi \big\|_{L^2(\mathbb{R}^d)} \\
    &\quad\leq C_{d,p}\exp\big( C_{d,p}T \| \phi\|_{L^2(\mathbb{R}^d)}^{\frac{4p}{4-dp}} \big) \\
    &\qquad \times\Big( \| \phi - P(\widetilde{\tau})\phi \|_{L^{2}(\mathbb{R}^d)} + \widetilde{\tau}^{\frac{4-dp}{8}} T\|\phi \|_{L^{2}(\mathbb{R}^d)}^{C_{d,p}} + \Big(\frac{\tau}{\widetilde{\tau}}\Big)^{\frac{1}{2}} \big( \|\phi\|_{L^2(\mathbb{R}^d)} +\|\phi\|_{L^2(\mathbb{R}^d)}^{p+1}\big) \Big)
    \end{aligned}
    \end{equation}
for any $0<\tau<\min\big\{\frac{c_{d,p}}{2}\|\phi\|^{-\frac{4p}{4-dp}}_{L^2(\mathbb{R}^d)},1\big\}$ and $\tau\leq \widetilde{\tau}<1$.

We omit the trivial case  $\tau\ge\min\big\{\frac{c_{d,p}}{2}\|\phi\|^{-\frac{4p}{4-dp}}_{L^2(\mathbb{R}^d)},1\big\}$.
Thus the convergence estimate \eqref{main-result5} is finally established.

\section{Appendix}\label{sec_Appen}

In this section, we give a proof of \eqref{NLS-bdd} for readers' convenience.
The argument is almost the same as the standard one.

\begin{proof}[Proof of \eqref{NLS-bdd}]
The argument of this proof is a minor modification of the proof of the time-continuous norm $\|u \|_{L^q(I_0; L^{r} )}$.
We give a proof briefly.
Let
    $$
    T_{loc}=c_{d,p} \|\phi\|_{L^2(\mathbb{R}^d)}^{-\frac{4p}{4-dp}}\quad\mbox{and}\quad I_0= [0,T_{loc}],
    $$
where $c_{d,p}>0$ is a constant to be determined below.
By Duhamel's principle \eqref{Duhamel_u} and the Strichartz estimates \eqref{Str_homo} and \eqref{Str_inhomo1}, we get
    $$
    \big\| P(\tau)u(n\tau)\phi \big\|_{\ell^q(I_0; L^{r} )}
    \leq C_d \| \phi \|_{L^2(\mathbb{R}^d)}
    + C_d \big\| |u|^p u \big\|_{L^{q_0'}(I_0; L^{r_0'} )},
    $$
where $(q_0,r_0)$ is defined as in \eqref{admisible_q0r0}.
By H\"older's inequality, the second term is bounded by
    $$
    \big\| |u|^p u \big\|_{L^{q_0'}(I_0; L^{r_0'})}
    \leq C_{d,p} T_{loc}^{1-\frac{dp}{4}} \big\| |u|^p \big\|_{L^{\frac{q_0}{p}}(I_0; L^{\frac{r_0}{p}})}
    \| u \|_{L^{q_0}(I_0; L^{r_0})}.
    $$
From this estimate and the result \eqref{NLS-bdd} for the time-continuous norm $\|u \|_{L^q(I_0; L^{r} )}$, we can take a small constant $c_{d,p}>0$ in $T_{loc}$ such that
    \begin{equation}\label{Apendix_1}
    \begin{aligned}
    \big\| P(\tau)u(n\tau)\phi \big\|_{\ell^q(I_0; L^{r} )}
    &\leq C_d \| \phi \|_{L^2(\mathbb{R}^d)}
    + C_{d,p} T_{loc}^{1-\frac{dp}{4}} \| u \|_{L^{q_0}(I_0; L^{r_0})}^{p+1} \\
    &\leq C_d \| \phi \|_{L^2(\mathbb{R}^d)}
    + \| \phi \|_{L^2(\mathbb{R}^d)}
    \end{aligned}
    \end{equation}
for all admissible pair $(q,r)$.
This implies \eqref{NLS-bdd} for the time-discrete norm $\big\| P(\tau)u(n\tau)\phi \big\|_{\ell^q(I_0; L^{r} )}$.
\end{proof}


\section*{Acknowledgments}
H. J. Choi was supported by the National Research Foundation of Korea (grant RS-2023-00280065) and Education and Research Promotion program of KOREATECH in 2022.
S. Kim was supported by the National Research Foundation of Korea (grant NRF-2022R1F1A1063379, RS-2023-00217116) and Korea Institute for Advanced Study(KIAS) grant funded by the Korea government (MSIP).
Y. Koh was supported by the the research grant of the Kongju National University in 2023 and the Visiting Professorship at Korea Institute for Advanced Study(KIAS).


\end{document}